\documentclass[ejs, preprint]{imsart}

\RequirePackage[numbers]{natbib}

\usepackage{amsmath,amsfonts,amssymb,amsthm,mathtools}
\usepackage[T1]{fontenc} 
\usepackage[utf8]{inputenc}
\usepackage{paralist}
\usepackage{bm}
\usepackage{verbatim}
\usepackage{hyperref}
\hypersetup{
    colorlinks=true,
    linkcolor=blue,
    filecolor=black,      
    urlcolor=black,
    citecolor = blue,
}

\numberwithin{equation}{section}
\allowdisplaybreaks

\newcommand{\dlambda}{\ensuremath{\lambda\mkern-9mu\lambda}}

\def\lf{\left\lfloor}   
\def\rf{\right\rfloor}
\newcommand{\flo}[1]{\lf #1 \rf}
\newcommand{\abs}[1]{\lvert #1 \rvert}
\newcommand{\norm}[1]{\lVert #1 \rVert}

\newcommand{\mb}[1]{\ensuremath{ \mathbb{#1}}}
\newcommand{\mc}[1]{\ensuremath{ \mathcal{#1}}}

\newcommand{\Pm}{\ensuremath{ \mathbb{P}}}
\newcommand{\dPm}{\ensuremath{ \operatorname{d}\!\mathbb{P}}}

\newcommand{\diff}{\ensuremath{ \operatorname{d}\!}}

\DeclareMathOperator{\Exp}{\mathbb{E}}
\DeclareMathOperator{\Var}{Var}
\DeclareMathOperator{\Cov}{Cov}

\newcommand{\ind}[1]{\ensuremath{ \mathbf{1} \! \left\lbrace #1 \right\rbrace  }}

\newcommand{\GEV}[1]{\ensuremath {\operatorname{GEV}\!\left(#1 \right)} }

\newcommand{\Norm}[1]{\ensuremath {\mc{N}\!\left(#1 \right)} }

\newcommand{\on}[1]{\ensuremath {\operatorname{#1} }}

\newcommand\restr[2]{{
  \left.\kern-\nulldelimiterspace 
  #1 
  \vphantom{\big|} 
  \right|_{#2} 
  }}

\newcommand\N{\mathbb{N}}
\newcommand\Z{\mathbb{Z}}

\newcommand\R{\mathbb{R}}

\newcommand{\wconv}{ \rightsquigarrow }

\newcommand{\dbl}{\ensuremath {\operatorname{db} }}
\newcommand{\sbl}{\ensuremath {\operatorname{sb} }}
\newcommand{\mbl}{\ensuremath {\operatorname{mb} }}

\newcommand{\ndb}{n_{\operatorname{db}}}
\newcommand{\nsb}{n_{\operatorname{sb}}}
\newcommand{\nmb}{n_{\operatorname{mb}}}

\newcommand{\Idb}{{I_{n}^{\dbl}}}
\newcommand{\Isb}{{I_{n}^{\sbl}}}
\newcommand{\Imb}{{I_{n}^{\mbl}}}
\newcommand{\Jdb}{{J_{n}^{\dbl}}}
\newcommand{\Jsb}{{J_{n}^{\sbl}}}
\newcommand{\Jmb}{{J_{n}^{\mbl}}}
\newcommand{\JJdb}{{\mc{J}_{n,\dbl}}}
\newcommand{\JJsb}{{\mc{J}_{n,\sbl}}}
\newcommand{\Unmb}{U_{n,r}^{\mbl}}

\newcommand{\Unsb}{U_{n,r}^{\sbl}}
\newcommand{\tilUnmb}{U_{n,r,\bm Z}^{\mbl}}

\newcommand{\e}{\ensuremath {\operatorname{e} }}

\theoremstyle{plain}
\newtheorem{theorem}{Theorem}[section]
\newtheorem{proposition}[theorem]{Proposition}
\newtheorem{lemma}[theorem]{Lemma}
\newtheorem{corollary}[theorem]{Corollary}

\theoremstyle{remark}
\newtheorem{condition}[theorem]{Condition}

\newtheorem{remark}[theorem]{Remark}
\newtheorem{example}[theorem]{Example}

\begin{document}

\begin{frontmatter}

\title{Limit theorems for non-degenerate U-statistics of block maxima for time series\support{Financial support by the German Research Foundation (DFG) is gratefully acknowledged (grant number 465665892).}}
\runtitle{U-statistics for block maxima}

\author{\fnms{Axel} \snm{Bücher}\corref{}\ead[label=e1]{axel.buecher@hhu.de}}
\and
\author{\fnms{Torben} \snm{Staud}\ead[label=e2]{torben.staud@hhu.de}}
\address{Heinrich-Heine Universität Düsseldorf \\
Mathematisches Institut \\
\printead{e1,e2}}

\runauthor{A. Bücher and T. Staud}

\begin{abstract}
The block maxima method is a classical and widely applied statistical method for time series extremes. It has recently been found that respective estimators whose asymptotics are driven by empirical means can be improved by using sliding rather than disjoint block maxima. Similar results are derived for general non-degenerate U-statistics of arbitrary order, in the multivariate time series case. Details are worked out for selected examples: the empirical variance, the probability weighted moment estimator and Kendall's tau statistic. The results are also extended to the case where the underlying sample is piecewise stationary. The finite-sample properties are illustrated by a Monte Carlo simulation study.
\end{abstract}



\end{frontmatter}

\section{Introduction}

A common target parameter in various domains of application is the distribution of componentwise yearly or seasonal maxima calculated from some underlying multivariate time series \citep{Kat02, BeiGoeSegTeu04}.
Statistical inference on the target distribution typically involves the assumption that the block maximum distribution is an extreme value distribution. The latter is justified by probabilistic results from extreme value theory: under broad conditions on the time series, the only possible limit distribution of affinely standardized componentwise maxima, as the block size goes to infinity, are extreme value distributions; see \cite{Lea83} for the univariate case and \cite{Hsi89} for multivariate extensions. 

The statistical literature on estimation and testing for extreme-value distributions is abundant, ranging from univariate estimators for the parameters of the generalized extreme value distribution \citep{PreWal80, HosWal85} to nonparametric estimators for extreme value copulas \citep{GenSeg09} and parametric estimators for max-stable process models \citep{PadRibSis10}.

Mathematically, statistical methods are typically validated under the additional assumption that the block maxima sample is serially independent. However, heuristically, both the independence assumption as well as the assumption that block maxima genuinely follow an extreme-value distribution should only be satisfied asymptotically, for the block size tending to infinity. \cite{Dom15, FerDeh15, DomFer19} have shown that specific univariate estimators are consistent and asymptotically normal in a sampling scheme where the block size tends to infinity, while maintaining an i.i.d.\ assumption on the underlying time series. For specific univariate and multivariate estimators, \cite{BucSeg14, BucSeg18a} also relax the i.i.d.\ assumption, and allow for more general stationary time series satisfying certain mixing conditions. 
It has moreover been found that estimators based on block maxima may be made more efficient by considering sliding rather than disjoint block maxima, both in the univariate \citep{RobSegFer09, BucSeg18-sl, BucZan23} and in the multivariate case \citep{ZouVolBuc21}.

In general, the field of asymptotic statistics is based on a number of fundamental theoretical tools like the central limit theorem, the delta-method, the empirical process or the concept of U-statistics \citep{Van98}. 
While the efficiency gain of the sliding block maxima method over its disjoint blocks counterpart mentioned in the previous paragraph has been established for classical empirical means as well as empirical (copula) processes, it has not been studied yet for the case of U-statistics. The present paper aims at filling this gap by studying non-degenerate U-statistics of disjoint and sliding block maxima samples. The topic is related but different to \cite{Oor23}, who study U-statistics in the univariate case where the kernel of order $m$ is evaluated blockwise in the largest $m$ order statistics of a (disjoint) block of observations.

In general, U-statistics comprise a number of important estimators like the empirical covariance, Wilcoxon's statistic or Kendall's tau statistic. A prominent example from extremes is the probability weighted moment estimator \citep{HosWal85}. Mathematical theory for i.i.d.\ random variables dates back to \cite{hoeffdingzl}; since then, several favorable statistical properties have been demonstrated \citep[Chapter 12]{Van98}. Asymptotic results on U-statistics have also been generalized to the time series context \citep{sen72, Yoshi76, DehWen10a}; unbiasedness then only holds asymptotically. 

The main result of this paper is Theorem~\ref{thm:U_conv}, where we establish a central limit theorem on the estimation error of U-statistics for multivariate disjoint and sliding block maxima under mild assumptions on the serial dependence and the kernel function. 
As in the papers mentioned before, the disjoint blocks version is found to be at most as efficient as the sliding blocks version. 
In selected examples, it is in fact found to be less efficient. 
The results are extended to a sampling scheme involving piecewise stationarities which is used to capture certain applications from environmental extremes where maxima are calculated based on, for example, summer days \citep{BucZan23}. 
The model is interesting mathematically, because unlike the disjoint block maxima sample the sliding block maxima sample is not stationary anymore.

The remaining parts of this paper are organized as follows: the underlying model assumptions and the definition of respective U-Statistics for disjoint and sliding block maxima are presented in Section \ref{sec:u-stat}. The main limit results are discussed in \ref{sec:asy}, and illustrated for three selected examples in Section~\ref{sec:examples}. Extensions to piecewise stationary time series are presented in Section~\ref{sec:ext}.  Results from a Monte Carlo simulation study illustrate the behavior in finite-sample situations (Section~\ref{sec:sim_study}). Finally, the proofs are deferred to Sections~\ref{sec:proofs}. Additional limit results under strong mixing assumptions and lengthy calculations of some asymptotic variances are postponed to a supplement.

\section{U-statistics of block maxima}
\label{sec:u-stat}

Recall the Generalized Extreme Value (GEV) distribution with parameters $\mu$ (location), $\sigma$ (scale) and $\gamma$ (shape), defined by its cumulative distribution 
function 
\[
G_{(\mu, \sigma, \gamma)}(x) := \exp\Big[-\Big\{1+\gamma\Big(\frac{x-\mu}{\sigma}\Big)\Big\}^{-\frac{1}{\gamma}}\Big] , \qquad 1+\gamma\,  \frac{x-\mu}{\sigma} >0.
\]
If $\eta=(\mu,\sigma,\gamma)' = (0,1,\gamma)'$, we will use the abbreviation $G_{(0,1,\gamma)} = G_\gamma$. The support of $G_\gamma$ is denoted by $S_\gamma=\{x \in \R: 1+\gamma x>0\}$.

An extension of the classical extremal types theorem to strictly stationary time series \citep{Lea83} implies that, under suitable broad conditions, affinely standardized maxima extracted from a stationary time series converge to the GEV-distribution. This was generalized to the multivariate case in \cite{Hsi89}, where the marginals are necessarily GEV-distributed. We make this an assumption, and additionally require the scaling sequences to exhibit some common regularity inspired by the max-domain of attraction condition in the i.i.d.\ case \citep{deextreme}. 

\begin{condition}[Multivariate Max-domain of attraction] \label{cond:mda}
Let $(\bm X_t)_{t\in\Z}$ denote a stationary time series in $\R^d$ with continuous margins. There exist sequences $(\bm a_r)_r=(a_{r}^{(1)}, \dots, a_{r}^{(d)}))_{r} \subset (0,\infty)^d, (\bm b_r)_r = ( b_{r}^{(1)}, \dots, b_{r}^{(d)})_r \subset \R^d$  and $\bm \gamma = (\gamma^{(1)}, \dots, \gamma^{(d)}) \in\R$, such that, for any $s>0$ and $j\in\{1, \dots, d\}$,
\begin{align}\label{eq:rvscale}
	\lim\limits_{r \to \infty}\frac{a_{\lfloor rs \rfloor}^{(j)}}{a_{r}^{(j)}}  = s^{\gamma^{(j)}}
	\qquad 
	\lim\limits_{r \to \infty} \frac{b_{\lfloor rs \rfloor }^{(j)} -b_{r}^{(j)}}{a_{r}^{(j)}} = \frac{s^{\gamma^{(j)}} -1}{\gamma^{(j)}}, 
	\end{align} 
	where the second limit  is interpreted as $\log(s)$ if $\gamma^{(j)}=0$. Moreover,   	
	for $r \to \infty$, 			
	\begin{align}\label{eq:firstorder} 
	\bm Z_r = (Z_{r}^{(1)}, \dots, Z_{r}^{(d)})  \ {\longrightarrow}_d \ \bm Z \sim G
 \end{align}
 where $G$ denotes a $d$-variate extreme-value distribution with marginal c.d.f.s $G_{\gamma^{(1)}}, \dots, G_{\gamma^{(d)}}$ and where
 \[
 Z_{r}^{(j)} = \frac{ \max(X_1^{(j)}, \ldots, X_r^{(j)}) - b_{r}^{(j)}}{a_{r}^{(j)}}, \quad j \in \{1, \dots, d\}.
 \]
\end{condition}

In the case $d\ge 2$,
let $C$ denote the unique extreme value copula associated with $\bm{Z}$. As is well-known, 
$C$ can be written as
\begin{align} \label{eq:evc}
C(\bm u) = \exp\left( -L(-\log u^{(1)}, \ldots, -\log u^{(d)}) \right)
\end{align}
for some \textit{stable dependence function} $L \colon [0,\infty]^d \to [0,\infty]$, which satisfies 
\smallskip
\begin{compactenum}
\item[(L1)] $L$ is homogeneous: $L(s \bm x ) = sL(\bm x)$ for all $s > 0$ and all $\bm x \in [0,\infty]^d$;
\item[(L2)] $L(\bm e_j) = 1$ for $j=1,\ldots,d,$ where $\bm e_j$ denotes the $j$-th unit vector in $\R^d;$
\item[(L3)] $\max\left(x^{(1)}, \ldots, x^{(d)}\right) \leq L(\bm{x}) \leq x^{(1)} + \ldots + x^{(d)}$ for all $\bm{x} \in [0, \infty]^d;$
\item[(L4)] $L$ is convex;
\end{compactenum}
\smallskip
see, e.g., \cite{GudSeg10, Pic81}.

Note that (\ref{eq:rvscale}) and (\ref{eq:firstorder}) may for instance be deduced from Leadbetter's $D(u_n)$ condition, a domain-of-attraction condition on the associated i.i.d.\ sequence with stationary distribution equal to that of $\bm X_0$ and a weak requirement on the convergence of the c.d.f.\ of $\bm Z_r,$ see Theorem 10.22 in \cite{BeiGoeSegTeu04}. 

From now on, we assume to  observe $\bm X_1,\ldots, \bm X_n$, an excerpt from a strictly stationary $d$-dimensional time series $(\bm X_t)_t$ satisfying Condition \ref{cond:mda} (some generalizations will be discussed in Section~\ref{sec:ext}).
For block size parameter $r \ll n$, define componentwise block maxima of size $r$ by
\begin{align*}
\bm{M}_{r,i} := \Big( M_{r,i}^{(1)}, \ldots, M_{r,i}^{(d)} \Big), \qquad 
&M_{r,i}^{(j)} := \max \Big\{ X^{(j)}_{i}, \ldots, X^{(j)}_{i+r-1} \Big\},
\end{align*}
where $ 
i\in\{1, \dots, {n-r+1}\}$ denotes the first observation within each block.

The traditional block maxima method is based on applying statistical methods to the sample of disjoint block maxima. The latter is given by $\mc{M}_{n,r}^{(\dbl)} = \big( \bm M_{r,i} \colon i \in \Idb \big)$, where 
$\Idb := \lbrace (i-1)r+1 \colon 1 \leq i \leq m \rbrace$ with $m = m_n := \flo{n/r}$. Note that $m$ is the number of disjoint blocks of size $r$ that fit into the sampling period.
Under Condition~\ref{cond:mda}, the sample of disjoint block maxima is stationary and approximately follows the multivariate extreme value distribution $G$.

Instead of partitioning the observation period into disjoints blocks, one may alternatively  slide the blocks through the observation period, thereby taking successive maxima of only one to the right instead of $r$. The resulting sliding block maxima sample is given by $\mc{M}_{n,r}^{(\sbl)} = \big( \bm M_{r,i} \colon i \in \Isb \big)$, where $\Isb := \lbrace 1, \ldots, n-r+1 \rbrace$. Under Condition~\ref{cond:mda}, the sliding block maxima sample is stationary as well, with approximate c.d.f.\ $G$. Hence, statistical methods that are based on estimating unknown expectations by empirical means are meaningful.

The case of classical empirical means has been treated in varying generality in \cite{BucSeg18-sl, ZouVolBuc21, BucZan23}. 
It was found that estimators based on sliding block maxima are typically more efficient than their disjoint block maxima counterparts, despite the fact that the sample $\mc{M}_{n,r}^{(\sbl)}$ is heavily dependent over time, even if $\bm (\bm X_t)_t$ is an i.i.d.\ sequence. 
In this paper we generalize these results to U-statistics of order $p \in \N$, with $p=1$ corresponding to classical empirical means. 

More precisely, let $h \colon (\R^{d})^p \to \R$ be a known symmetric measurable function of $p$ $d$-dimensional input variables, subsequently referred to as a kernel of order $p$. 
The main objects of interest in this paper are the associated U-statistic of order $p$ given by, for $\on{mb} \in \{ \on{db}, \on{sb}\}$,
\begin{equation}
\label{eq:u_stat}
U_{n,r}^{\mbl} 
:= 
\Unmb(h) 
:= 
\binom{\nmb}{p}^{-1} \sum_{(i_1,\dots, i_p) \in \Jmb} h\!\left( \bm M_{r,i_1}, \dots, \bm M_{r,{i_p}}\right),
\end{equation}
where $\nmb = |\Imb|$ denotes the length of the block maxima sample (i.e., $\ndb = m$ if $\on{mb} = \on{db}$ and  $\nsb = n-r+1$ if $\on{mb} = \on{sb}$) and where 
\[
\Jmb := \Jmb(p) :=\lbrace (i_1, \dots, i_p) \in (\Imb)^p \colon i_1< \dots < i_p \rbrace.
\]
A standard heuristic argument suggests that, for the majority of summands 
in \eqref{eq:u_stat}, the underlying block maxima can be considered as asymptotically independent.
As a consequence, $U_n^{\mbl}$ should be considered as an estimator for
\begin{align}
\theta_r = \theta_r(h)
&:= \nonumber 
\int \dots \int h(\bm x_1, \dots, \bm x_p) \diff \Pm_{\bm M_{r,1}}(x_1) \dots \diff \Pm_{\bm M_{r,1}}(x_p) \\
&= \label{eq:thetar}
\Exp[h(\tilde{\bm M}_{r,1}^{(1)}, \dots, \tilde{\bm M}_{r,1}^{(p)})],
\end{align}
where $\tilde{\bm M}_{r,1}^{(1)}, \dots, \tilde{\bm M}_{r,1}^{(p)}$ are i.i.d.\ copies of $\bm M_{r,1}$.
We are interested in obtaining asymptotic results for the estimation error
\[
\Unmb(h)- \theta_r(h)
\]
in an asymptotic framework where $r=r_n \to \infty$ such that $r=o(n)$ for $n\to\infty$.

\section{Limit theorems for U-statistics of block maxima} \label{sec:asy}

We start by introducing further conditions and notations.
First, throughout the proofs we will use traditional blocking techniques relying on mixing coefficients. The latter are well-known to control the serial dependence of the underlying time series. A similar condition has been imposed in \cite{BucSeg14}, among others.

\begin{condition}[Block size and serial dependence]\label{cond:ser_dep}
For the block size sequence $(r_n)_n$ it holds that, as $n\to\infty$,
\begin{compactenum}
    \item[(a)] $r_n \to \infty$ and $r_n = o(n).$
    \item[(b)] There exists a sequence $(\ell_n)_n \subset \N$ such that $\ell_n \rightarrow \infty $, $\ell_n = o(r_n)$ and $ \frac{r_n}{\ell_n}  \alpha(\ell_n) = o(1)$ and $\frac{n}{r_n} \alpha(\ell_n)=o(1)$.
    \item[(c)] $\bigl(\frac{n}{r_n}\bigr)^{1+\omega} \beta(r_n) = o(1) $ for some $\omega > 0$.
\end{compactenum}
Here, $\alpha$ and $\beta$ denote the alpha- and beta-mixing coefficients, see \cite{Bra05} for exact definitions and basic properties. Subsequently, we often write $r=r_n$ and $\ell=\ell_n$.
\end{condition}

The expectation and higher order moments of $h(\bm M_{r,i_1}, \dots, \bm M_{r,i_p})$ in (\ref{eq:u_stat}) will be controlled by uniform integrablity and by relying on the convergence of rescaled block maxima from Condition \ref{cond:mda}. For that purpose, we need the kernel function $h$ to behave well under location-scale transformations; see also \cite{Seg01}, Chapter 5, and \cite{Oor23} for a similar, slightly more restrictive assumption.

\begin{condition}[Location-scale property of the kernel function] \label{cond:ker_traf}
There exist functions $f \colon (\R^{d})^p \to (0,\infty), \, \ell\colon (\R^{d})^p \to \R$ such that, for all $ \bm x_{1}, \ldots, \bm x_{p}, \bm b \in \R^d$ and  $\bm{a} \in {(0, \infty)}^d$,
\begin{equation} \label{eq:ker_traf}
h\!\left( \frac{\bm x_{1}- \bm b}{ \bm{a}}, \ldots, \frac{\bm  x_{p}- \bm b}{ \bm{a}} \right) = \frac{h(\bm x_{1}, \ldots, \bm x_{p})}{f( \bm{a}, \bm b)} - \ell(\bm{a}, \bm b),
\end{equation}
\end{condition}
where ${\bm x}/{\bm y} := \big( x^{(1)}/y^{(1)},\ldots,x^{(d)}/ y^{(d)} \big)$ for $\bm x \in \R^d, \bm y \in (0,\infty)^d.$

\begin{example}\label{ex:kernels}
Condition \ref{cond:ker_traf} is met for the following kernel functions. Note that the kernels in (5) to (7) may be used to construct tests for stochastic independence; see, for instance, \cite{leung_drton_kernels}. In the current case, this corresponds to testing asymptotic independence of the coordinates of $\bm X_1$.
\begin{compactenum}[(1)]
\item The mean kernel: $h(x) = x$ with $d = 1, \, p = 1, \,f(a,b) = a, \,\ell(a,b) = b.$ 
\item The variance kernel: $h(x,y) = (x-y)^2/2$  with $d = 1, \, p = 2, \, f(a,b) = a^2,\, \ell \equiv 0.$
\item Gini's mean difference kernel: $h(x,y) = |x-y|/2$ with $d=1, p = 2, f(a,b) = a, \ell \equiv 0.$
\item The modified probability weighted moment kernel of degree $k \in \N$ (see also Section~\ref{subsec:asy_pwm}):
$h_k(x_1,\ldots,x_{k}) = \max \{x_1, \dots, x_k\} / k$ with $d = 1, \, p = k, \,f(a,b) = a,\, l(a,b) = \frac1k\cdot \frac{b}a$. 
\item Kendall's $\tau$ kernel: $h(\bm x, \bm y) = \ind{(x^{(1)} - y^{(1)} ) 
(x^{(2)} - y^{(2)}) > 0}$ with $d = 2,\, p = 2,\, f\equiv 1,\, \ell \equiv 0.$ 
\item Spearman's $\rho$ kernel: $h(\bm{x}_1, \bm{x}_2, \bm{x}_3) = 2^{-1} \sum_{\pi \in \on{S}_3} \on{sgn}(x_{\pi_1}^{(1)} - x_{\pi_2}^{(1)}) \on{sgn}(x_{\pi_1}^{(2)} - x_{\pi_3}^{(2)})$  with $d=2, \, p = 3,\, f \equiv 1, \ell \equiv 0$ and where $\on{S}_n$ denotes the symmetric group of order $n$.
\item Hoeffding's $D$ kernel and Bergsma and Dassio's $t^\ast$ kernel: we refer to \cite{leung_drton_kernels} for the kernel definition, which satisfy $d=2, f\equiv1, \ell \equiv 0$ and $p=4$ and $p=5$, respectively. 
\end{compactenum}
\end{example}

From now on, for the ease of notation, we only consider the case $p = 2$ (see also \cite{DehWen10a}, among others). For $i \in \{1, \dots, n-r+1\}$, let
\[
\bm Z_{r,i} := \Big(Z_{r,i}^{(1)}, \ldots, Z_{r,i}^{(d)}\Big), \qquad
Z_{r,i}^{(j)} := (M_{r,i}^{(j)} - b_r^{(j)})/ a_r^{(j)}.
\]
with $\bm a_r$ and $\bm b_r$ from Condition \ref{cond:mda}. Note that $(\bm Z_{r,i})_i$ is stationary with $\bm Z_{r,1} \wconv G$ as $n\to\infty$. Further, 
under Condition \ref{cond:ker_traf} one has 
\begin{equation}
\label{eq:hZ_to_hM}
h\left(\bm M_{r,i}, \bm M_{r,j}\right)
=
f(\bm a_r, \bm b_r) \big\{ h\left(\bm Z_{r,i}, \bm Z_{r,j}\right) + \ell(\bm a_r, \bm b_r) \big\},   
\end{equation}
which will ultimately allow to deduce asymptotic results on $\Unmb$ defined in \eqref{eq:u_stat} from respective results on
\begin{equation}\label{eq:u_stat_resc}
U_{n, r, Z}^{\mbl} := {U}_{n,r,Z}^{\mbl}(h) := 
\binom{\nmb}{2}^{-1} \sum_{(i,j) \in \Jmb} h\!\left( \bm Z_{r,i}, \bm Z_{r,j} \right).
\end{equation}
Heuristically, the expectation of $U_{n, Z}^{\mbl}$ is close to 
\begin{align} \label{eq:varthetar}
\vartheta_r = \vartheta_r(h)
= \int \int h(\bm x,\bm y) \diff \Pm_{\bm Z_{r,1}}(\bm x) \diff \Pm_{\bm Z_{r,1}}(\bm y)
=\Exp[h(\bm Z_{r,1}, \tilde{\bm Z}_{r,1})]
\end{align}
with $\tilde{\bm Z}_{r,1}$ an independent copy of $\bm Z_{r,1}$. The sequence $\vartheta_r$ in turn converges to 
\begin{align}\label{eq:vartheta}
\vartheta := \Exp[h(\bm Z, \tilde{\bm Z})],
\end{align}
under suitable integrability assumptions; here $\bm Z, \tilde{\bm Z} \sim G$ are independent (see Lemma~\ref{lem:varthetar} below).
The necessary integrability condition, which will also ensure convergence of higher order moments, is as follows.

\begin{condition}[Asymptotic integrability] \label{cond:h_int_std}
There exists a $\nu > 2/\omega$ with $\omega$ from  Condition~\ref{cond:ser_dep} such that:
\begin{compactenum}
    \item[(a)] $\limsup_{r\to \infty} \int \int \abs{h(\bm{x}, \bm{y})}^{2+\nu} \on{d}\!\mathbb{P}_{\bm{Z}_{r,1}}(\bm{x}) \on{d}\!\mathbb{P}_{\bm{Z}_{r,1}}(\bm{y}) < \infty,$
    \item[(b)] $\limsup_{r\to \infty} \sup_{s \in \N} \int \abs{h(\bm{x}, \bm{y})}^{2+\nu} \on{d}\!\mathbb{P}_{(\bm{Z}_{r,1}, \bm{Z}_{r,1+s})}(\bm x, \bm{y}) < \infty$.
\end{compactenum}    
\end{condition}
\noindent Note that the two moment assumptions may be understood as an asymptotic formulation of uniform moments as used in \cite{DehWen10b}.
In many situations, the conditions are easily satisfied, see, e.g., Section \ref{sec:examples}. Finally, for kernels of higher order than $p=2$, more complicated versions of this condition will be needed, see \cite{Yoshi76}.

Additional notation is needed to formulate the asymptotic limit results for $U_{n,r}^{\mbl}$. Recall $G$ from Condition~\ref{cond:mda}. Let $L$ denote the stable tail dependence function of $G$ if $d\ge 2$, and the identity on $[0,\infty]$ if $d=1$.
For $\bm u, \bm v \in [0,1]^d$ and $\xi \ge 0$, let
\begin{align} \label{eq:cxi}
C_{\xi}(\bm u, \bm v) = \exp \Big[ -L_\xi \Big( -\log(u^{(1)}),\ldots, -\log(u^{(d)}), -\log(v^{(1)}), \ldots, -\log(v^{(d)}) \Big) \Big],
\end{align}
where, for $\bm x, \bm y \in [0,\infty]^d$, 
\begin{multline} \label{eq:lxi}
L_\xi(\bm{x},\bm{y})  
:=  (\xi \wedge 1) \cdot \Big \lbrace L(x^{(1)},\ldots,x^{(d)}) + L\left( y^{(1)},\ldots,y^{(d)} \right) \Big \rbrace \\
+ (1- (\xi \wedge 1)) \cdot L \left( x^{(1)} \vee y^{(1)},\ldots, x^{(d)} \vee y^{(d)} \right).
\end{multline}
As shown in Lemma~\ref{lem:overlap_wconv} below, $C_\xi$ defines a $2d$-variate extreme-value copula with stable tail dependence function $L_\xi$. Let $G_\xi$
 denote the $2d$-variate extreme value distribution with copula $C_\xi$ and margins $G_{\gamma^{(1)}}, \dots, G_{\gamma^{(d)}}, G_{\gamma^{(1)}}, \dots, G_{\gamma^{(d)}}$, i.e.,
\begin{align} \label{eq:gxi}
G_\xi(\bm x, \bm y) = C_\xi\Big\{G_{\gamma^{(1)}}(x^{(1)}), \dots, G_{\gamma^{(d)}}(x^{(d)}), G_{\gamma^{(1)}}(y^{(1)}), \dots, G_{\gamma^{(d)}}(y^{(d)}) \Big\}
\end{align}
for $\bm x, \bm y \in \R^d$.
Note that $G_{\xi}(\bm{x},\bm{y}) = G(\bm{x})G(\bm{y})$ for $\xi > 1$.
Further, $G_\xi$ is the multivariate analogue of $G_{\alpha_0, \xi}$ in  Formula (5.1) in \cite{BucSeg18-sl} and, in the case $d=1$, also appeared in Formula (13) in \cite{BucZan23}.

Finally, for 
$(\bm Z_{1,\xi}, \bm Z_{2,\xi}) \sim G_\xi$, let
\begin{align}\label{eq:as_vars}
\sigma^2_{\dbl} := 4\Var(h_1(\bm Z)), \qquad 
\sigma^2_{\sbl} := 8 \int_0^1 \Cov\big(h_1(\bm Z_{1,\xi}), h_1(\bm Z_{ 2,\xi})\big) \diff \xi .
\end{align}
where, 
\begin{align}\label{eq:h_1}
h_1 \colon \R^d \to \R, \, h_1(\bm z) &:= \Exp[h(\bm z, {\bm Z})]  - \vartheta
\end{align}
with $\bm Z \sim G$ and $\vartheta$ from \eqref{eq:vartheta}.
The following result is the main result of this paper.

\begin{theorem}\label{thm:U_conv}
Suppose Conditions \ref{cond:ser_dep}, \ref{cond:ker_traf} and \ref{cond:h_int_std} are met. Furthermore let $h$ be $\dlambda^{2d}$-a.e.\ continuous and bounded on compact sets. Then, for $\on{mb} \in \lbrace \on{db}, \on{sb} \rbrace$,  
\begin{equation*}
\frac{\sqrt{m}}{ f(\bm{a}_r, \bm{b}_r)} \cdot ( \Unmb  - \theta_r) \wconv 
\Norm{0, \sigma_{\on{mb}}^2},
\end{equation*}
with $\theta_r$ from \eqref{eq:thetar} and $\sigma^2_{\mbl}$ from \eqref{eq:as_vars}. Moreover, 
$\sigma_{\on{sb}}^2 \leq  \sigma_{\on{db}}^2$.
\end{theorem}

Note that, under Condition~\ref{cond:ker_traf},  $\theta_r=f(\bm{a}_r, \bm{b}_r) \{ \vartheta_r + \ell(\bm{a}_r, \bm{b}_r)\}$ with $\vartheta_r$ from \eqref{eq:varthetar}.
In certain situations (in particular when $\ell=0$ and $f \equiv \text{const}$; see, e.g., Kendall's tau), one may be willing to regard $\Unmb$ as an estimator for the asymptotic analogue
\begin{align} \label{eq:varthetat}
\tilde \vartheta_r = f(\bm{a}_r, \bm{b}_r) \{ \vartheta + \ell(\bm{a}_r, \bm{b}_r)\}.
\end{align}
For instance, in case of the variance kernel (see also Section~\ref{subsec:var}), $\tilde \vartheta_r$ is the variance of the GEV($b_r, a_r, \gamma$)-distribution, which is exactly the GEV-distribution approximating the distribution of $M_{r,1}$, see Assumption~\ref{cond:mda}.
Under an additional bias condition, we may deduce the following result on the estimation error.

\begin{corollary}\label{cor:U_conv_bias}
Additionally to the assumptions made in Theorem~\ref{thm:U_conv}, suppose that the limit $B =\lim_{n \to \infty} B_n$ exists,
where
\begin{align} \label{eq:bn} 
B_n:= 
\sqrt m( \vartheta_r - \vartheta).
\end{align}
Then, for $\on{mb} \in \lbrace \on{db}, \on{sb} \rbrace$,  
\begin{equation*}
\frac{\sqrt{m}}{f(\bm{a}_r, \bm{b}_r)} \cdot (\Unmb  - \tilde \vartheta_r ) \wconv 
\Norm{B, \sigma_{\on{mb}}^2},
\end{equation*}
with $\sigma^2_{\mbl}$ from \eqref{eq:as_vars} and $\tilde \vartheta_r$ from \eqref{eq:varthetat}.
\end{corollary}

\begin{remark}[Generalizations]
Using the Cramér-Wold Theorem it is possible to generalize the limit theorems to the case of joint convergence involving a finite number of kernel functions. Moreover, as mentioned before and at the cost of a more complicated notation, one might extend the results to higher kernel degrees $p \in \N$. Joint weak convergence then even holds for kernels of different degrees. These generalizations allow, for example, to handle the joint convergence of probability weighted moments estimators of different order, which would be needed to deduce the asymptotics of the PWM-estimator for the parameters of the GEV-distribution.
Further generalizations concerning different model assumptions are worked out in Section~\ref{sec:ext} and Section~\ref{sec:ext_alpha} in the supplement.
\end{remark}

\begin{remark}[A bias reduced version of the sliding blocks estimator]
In view of Lemma \ref{lem:overlap_wconv} below, the block maxima $\bm M_{r,i}$ and $\bm M_{r,j}$ are asymptotically independent for $|i-j| \ge r$ and asymptotically dependent otherwise. As a consequence, the summands $h(\bm M_{r,i}, \bm M_{r,j})$ with $|i-j|<r$ induce a \textit{dependency bias}, which suggests to replace $\Unsb$ by 
\[
\tilde U_{n,r}^{\sbl} := 
\binom{\tilde n_{\sbl}}{2}^{-1} \sum_{(i,j) \in \tilde J_n^{\sbl}} 
h(\bm M_{r,i}, \bm M_{r,j}),
\]
where $\tilde J_n^{\sbl} = \lbrace (i,j) \in (\Isb)^2 \colon i<j, j-i \ge r \rbrace$.
Note that $|J_n^{\sbl} \setminus \tilde J_n^{\sbl}|=O(nr)$, 
which can used to show that
$
\frac{\sqrt m}{ f(a,b)} (\tilde U_{n,r}^{\sbl} - U_{n,r}^{\sbl}) = O_{\mathbb P}\big((1+\ell(\bm a_r, \bm b_r)) m^{-1/2}\big)
$
with $\ell$ from Condition~\ref{cond:ker_traf}. Hence, the two estimators are typically asymptotically equivalent.
Throughout, we only consider $U_{n,r}^{\sbl}$ for simplicity.
\end{remark}

\section{Examples}
\label{sec:examples}

Details are worked out for specific kernel functions of interest.

\subsection{Variance estimation} \label{subsec:var}

The variance is one of most fundamental parameters to describe a distribution of interest, which, in our case, is $\sigma_r^2 := \Var(M_{r,1})$. The respective empirical variance, based on either disjoint or sliding block maxima, is given by
\[
\hat{\sigma}^2_{n,r, \mbl} 
=
\frac{1}{\nmb -1}\sum_{i \in \Imb} \big( M_{r,i} - \overline{M}_{r}^{\mbl} \big)^2, \qquad \mbl \in \{ \dbl, \sbl\},
\] 
where $\overline{M}_{r}^{\mbl}:= \nmb^{-1} \sum_{i \in \Imb} M_{r,i}$. As is well-known, the empirical variance can be written as a U-statistic of order $p=2$, that is, 
\[
\hat{\sigma}^{2}_{n,r, \mbl}  = \Unmb(h_{\mathrm{Var}}), \qquad h{_\mathrm{Var}}(x,y)= (x-y)^2/2.
\]
The following result is a direct consequence of Theorem~\ref{thm:U_conv}.

\begin{corollary}\label{cor:var_est}
Suppose Condition \ref{cond:ser_dep} is met with $\gamma < 1/4$ and that there exists a $\nu > 2/\omega$ such that $\limsup_r \mb{E} |Z_{r,1}|^{4+\nu} < \infty$. Then 
\begin{equation*}
\frac{\sqrt{m}}{a_r^2} \Big( \hat{\sigma}^2_{n,r, \mbl} - \sigma^2_r \Big)
\wconv \Norm{0, \sigma^2_{\mbl}},
\end{equation*}
where $\sigma^2_{\dbl}$ and $\sigma^2_{\sbl}$ only depend on the tail index $\gamma$. Explicit formulas are provided in (\ref{eq:var_est_asyvar_db}) and (\ref{eq:var_est_asyvar_sb}) in the supplement, respectively.  Moreover, $\sigma^2_{\sbl} < \sigma^2_{\dbl}$.
\end{corollary}

The assumption $\gamma<1/4$ is natural, as asymptotic normal results on empirical variances require finite fourth moments; in the case of the GEV-distribution, this exactly corresponds to $\gamma<1/4$.
Figure \ref{fig:asy_var_var_ratio} shows the ratio of the asymptotic variances, $\sigma^2_{\dbl} / \sigma^2_{\sbl}$ as a function of $\gamma$. We observe that the estimator based on sliding blocks has a significantly smaller variance for negative $\gamma$, say $\gamma < -.25$, while hardly any difference is visible for positive $\gamma$.

\begin{figure}[t!]
\centering
\makebox{\includegraphics[width=0.7\textwidth]{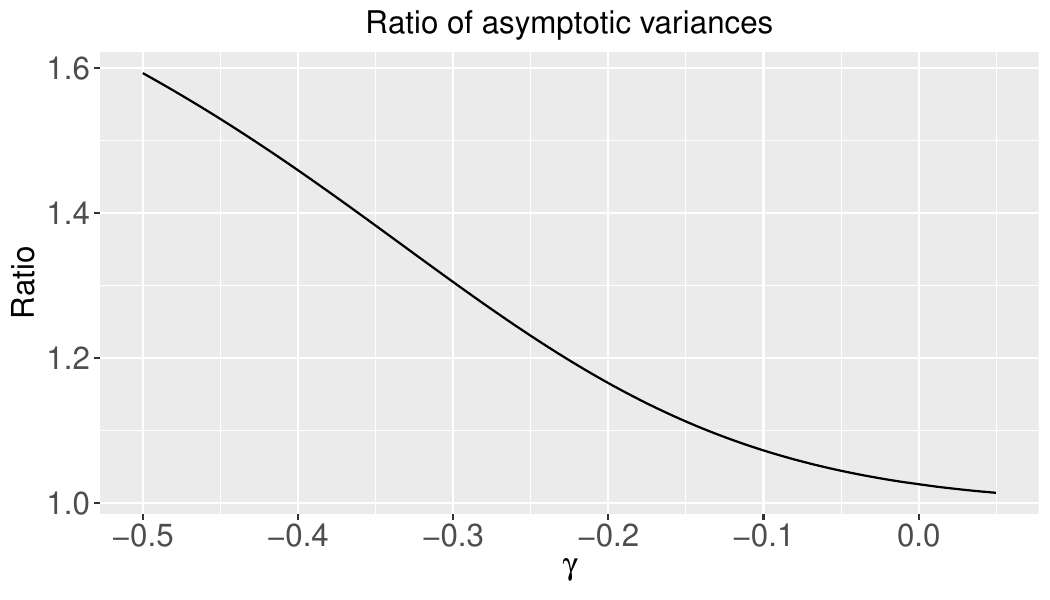}}\vspace{-.3cm}
\caption{
Graph of $\gamma \mapsto \sigma^2_{\dbl}/\sigma^2_{\sbl}$ with $\sigma^2_{\mbl}$ as in \eqref{eq:var_est_asyvar_db} and  \eqref{eq:var_est_asyvar_sb}.
}\label{fig:asy_var_var_ratio}	
\vspace{-.3cm}
\end{figure}

The previous results may be made more explicit when imposing a specific time series model. We exemplary  work out details for a marginal transformed version of the ARMAX-model. The model is defined as follows: for an i.i.d.\ sequence $(W_t)_{t \in \Z}$ of Fréchet(1) distributed random variables and $\alpha \in (0,1]$, consider the ARMAX(1) recursion defined as
\begin{equation}\label{eq:armax}
Y_t = \max \left \lbrace \alpha Y_{t-1}, \, (1-\alpha) W_t \right\rbrace, \quad t \in \Z.
\end{equation}
The recursion has the stationary solution $
Y_t := \max_{j \geq 0} (1-\alpha) \alpha^j W_{t-j}$, which has Fréchet(1) distributed marginals and extremal index $\theta=1-\alpha$, see Example 10.5 in \cite{BeiGoeSegTeu04}.
Define $X_t$ as the transformed random variables $X_t := F_{\gamma}^{\leftarrow}(F_W(Y_t)),$ where $F_W$ is the c.d.f. of a Fréchet(1) distribution, $F_\gamma$ is the c.d.f.\ of the Pareto family defined as
\begin{align} \label{eq:gpd}
F_\gamma(x) := 
\begin{cases}
\left( 1- (1+\gamma x)^{-1/\gamma} \right) \ind{x \geq 0}, \quad &  \gamma > 0\\
\left( 1- (1+\gamma x)^{-1/\gamma} \right) \ind{ 0 \leq x \leq -1/\gamma}, \quad &\gamma < 0\\
\left( 1- \exp(-x) \right) \ind{x \geq 0}, & \gamma = 0
\end{cases},
\end{align}
and where $F^{\leftarrow}$ is the left continuous generalized inverse of $F$. 
By \cite{BerBuc18} and \cite{Bra05} the untransformed time series $(Y_t)_t$ is exponentially $\beta$-mixing, which implies the same for $(X_t)_t$. This results in a large spectrum of choices for $r_n$ and $\ell_n$ satisfying Condition~\ref{cond:ser_dep}, which can hence be regarded as non-restrictive. 
We will prove in Section~\ref{sec:proofs_ex} that, if $\gamma<1/4$ and if $r= o(n), n = o(r^3)$, all assumptions from Corollary~\ref{cor:var_est} are met, with $a_r = (r(1-\alpha))^\gamma$ and  $b_r = \lbrace (r(1-\alpha))^\gamma -1 \rbrace / \gamma$. Hence,
\begin{equation}
\label{eq:var-armax}
 \frac{\sqrt{m}}{(r(1-\alpha))^{2\gamma}} ( \hat{\sigma}^2_{n,r, \mbl} - \sigma_r^2 )
 \wconv \Norm{0, \sigma^2_{\mbl}}
\end{equation}
as asserted.
Moreover, one may show that bias condition is met with $B=0$, whence $\sigma_r^2$ may be replaced by $(r(1-\alpha))^{2\gamma} \tau^2(\gamma)$, where 
$\tau^2(\gamma) := 
\frac1\gamma \{ \Gamma(1-2\gamma)- \Gamma(1-\gamma)^2\} \bm 1(\gamma \in (-\infty, 1/2) \setminus \{0\} ) + (\pi^2/6) \bm 1( \gamma = 0)   
$
is the variance of the $\GEV{\gamma}$ distribution, with
$\Gamma(x) := \int_0^\infty t^{x-1}e^{-t} \on{d}\!t,\, x > 0$, the Gamma function.

\subsection{The probability weighted moment estimator}
\label{subsec:asy_pwm}

Let $M \sim G_\eta$ be a GEV-distributed random variable with parameter $\eta = (\mu,\sigma, \gamma)^\prime \in \R\times(0,\infty)\times (-\infty,1).$ For $k\in \N_0$, the $k$th probability weighted moment (PWM) of $M$ is given by
\begin{equation}\label{eq:pwm}
\beta_{\eta,k} := \Exp[MG_{\eta}^k(M)].
\end{equation}
It is well-known that $\eta$ is a one-to-one function of the first three probability weighted moments \citep{HosWal85}. Replacing the moments in (\ref{eq:pwm}) by suitable estimators and plugging those into the one-to-one function results in (the) PWM estimator for $\eta$. One version, as proposed in \cite{pwm_est}, is given by
\begin{equation}
\tilde{\beta}_0 := \frac{1}{n}\sum_{i=1}^n M_i, \quad 
\tilde{\beta}_k := \frac{1}{n} \sum_{i=1}^n \frac{(i-1)\cdot \ldots \cdot (i-k)}{(n-1)\cdot \ldots \cdot (n-k)}M_{(i)}, \quad k \geq 1
\end{equation}
where $\mc{M} = (M_1,\ldots,M_n)$ is a sample of random variables distributed as $M$ and $M_{(1)} \leq \ldots \leq M_{(n)}$ is the ordered sample. If $\mc{M}$ is an i.i.d.\ sample, then there are no ties with probability 1, whence $\tilde \beta_{k} = \hat \beta_k$, where, for $k\in\N$, 
\begin{equation}\label{eq:pwm_ust}
\hat{\beta}_{k-1} = \binom{n}{k}^{-1} \sum_{1 \leq i_1 < \ldots < i_k \leq n}
h_{\on{pwm},k}(M_{i_1},\ldots, M_{i_k})
\end{equation}
with the permutation invariant kernel function
\begin{equation}\label{eq:pwm_kern}
h_{\on{pwm}, k}(x_1,\ldots, x_k) := \frac{1}{k} \sum_{j=1}^{k} \bm 1 \Big\{ \max_{1 \leq i \leq k, i \neq j} x_i \leq x_j \Big\} x_j
\end{equation}
Clearly, $\hat{\beta}_{k-1}$ is a U-statistic of order $k$ that is unbiased for $\Exp[\hat{\beta}_{k-1}] = \beta_{\eta,k-1}$ in case the sample is i.i.d.

In this section, we apply Theorem~\ref{thm:U_conv} to derive limit results for the estimator
\[
\hat \beta_{k-1}^{\mbl} = U_{n,r}^{\mbl}(h_{\on{pwm},k})
\]
with $\mbl \in \{\dbl, \sbl\}$. 
For simplicity, we restrict attention to the case $k=2$, which yields a U-statistic of order $2$. Since the function $h_{\on{pwm},2}$ does not satisfy Condition~\ref{cond:ker_traf}, we will need the modified kernel function $\tilde{h}_{\on{pwm},2}(x,y) := \max(x,y)/2$ from Example~\ref{ex:kernels}.

\begin{proposition}\label{ex:pwm}
Suppose that $(X_t)_{t \in \Z}$ from Condition~\ref{cond:mda} does not contain ties with probability 1 and that Condition \ref{cond:ser_dep} is met. If there exists $\nu > 2/\omega$ such that $\limsup_{r\to\infty} \mb{E} |Z_{r,1}|^{2+\nu} < \infty$, then, for $\mbl \in \{\dbl, \sbl\}$,  
\[
\frac{\sqrt{m}}{a_r} \left \lbrace U_{n,r}^{\mbl}(h_{\on{pwm},2}) -  U_{n,r}^{\mbl}(\tilde{h}_{\on{pwm},2}) \right \rbrace 
\xrightarrow[]{L^2} 0 
\]
If, moreover, the limit $B := \lim_n B_n$ exists, where $B_n = \sqrt{m} \Exp[F_r(Z_{r,1})Z_{r,1} - G_\gamma(Z)Z]$ with $F_r$ the c.d.f.\ of $Z_{r,1}$ and $Z \sim G_\gamma$, then
\[
\sqrt{m}\left( \frac{\hat{\beta}_1^{\mbl} - \beta_{(0,1,\gamma),1}}{a_r} \right) \wconv
\Norm{B, \sigma^2_{\mbl}}, 
\]
with $0<\sigma^2_{\sbl} < \sigma^2_{\dbl}.$ 
\end{proposition}

Note that similar asymptotics have also been worked out in \cite{BucZan23}, where the derivation was based on explicit expansions of the kernel function involving empirical cumulative distribution functions.
Comparing our result with their Theorem 3.5, we observe that our result is slightly more restrictive, since we impose $\beta$-mixing rather than $\alpha$-mixing. An extension to $\alpha$-mixing is given in Section~\ref{sec:ext_alpha} in the supplement.

\subsection{Estimation of Kendall's tau}
\label{subsec:tau}

Kendall's tau statistic is a well-known nonparametric distribution-free measure of rank correlation that quantifies the degree of association between two variables \citep{Kend38}. The population version $\tau=\tau(\bm X)$ for a bivariate vector $\bm X=(X^{(1)},X^{(2)})$ is defined as follows: for i.i.d.\ copies $\bm X_1, \bm X_2$ of $\bm X$, we have $\tau := \pi_c - \pi_d = 2\pi_c - 1,$ where $\pi_c := \Pm((X_1^{(1)} - X_2^{(1)})(X_1^{(2)} - X_2^{(2)}) > 0)$ and $\pi_d := \Pm((X_1^{(1)} - X_2^{(1)})(X_1^{(2)} - X_2^{(2)}) < 0)$ denote the probabilities of concordance and discordance of $\bm X_1, \bm X_2$, respectively. 
Applied to bivariate extreme value distributions, Kendall's tau provides a useful summary of extremal dependence; see \mbox{\cite[pp.~274-275]{BeiGoeSegTeu04}} and the references therein.

For a bivariate sample $(\bm{X_1}, \ldots, \bm{X}_n)$ Kendall's $\tau$-statistic can be written as $\hat \tau_n = \binom{n}{2}^{-1} \sum_{1 \leq i < j \leq n} \{ 2 h_\tau(\bm{X}_i, \bm{X}_j) - 1 \}$ with $h_\tau$ as in Example~\ref{ex:kernels}(5). 
For $\mbl \in \{\dbl, \sbl\}$, let $\hat \tau_{n,r}^{\mbl}$ denote Kendall's $\tau$-statistic applied to the sample of disjoint or sliding block maxima. An application of Theorem~\ref{thm:U_conv} yields the following result.

\begin{proposition}\label{ex:kend_tau}
Suppose Condition~\ref{cond:ser_dep} is met. Then, with $\tau_r := \tau(M^{(1)}_{r,1}, M^{(2)}_{r,1})$, we have, for $\mbl \in \{ \dbl, \sbl\}$,
\[
\sqrt{m} \big( \hat{\tau}^{\mbl}_{n,r} - \tau_r \big) \wconv \Norm{0, \sigma^2_{\mbl}},
\]
where the asymptotic variances can be represented as a function of the extreme-value copula $C$ from \eqref{eq:evc} as follows:
\begin{align*}
  \sigma^2_{\dbl} & = 16 \Bigg \lbrace \int_{[0,1]^2} \big\{ C(\bm u) + \bar C(\bm u) \big\}^2 \diff C(\bm{u}) - 4 \Big( \int_{[0,1]^2} C(\bm{u}) \diff C(\bm{u}) \Big)^2  \Bigg \rbrace\\ 
\\
\sigma^2_{\sbl} & = 
32 \int_0^1  \Bigg( \int_{[0,1]^2 \times [0,1]^2} \big\{ C(\bm u) + \bar C(\bm u) \big\}\big\{ C(\bm v) + \bar C(\bm v) \big\}\diff C_\xi(\bm u, \bm v) \\
& \hspace{6cm} - 4 \Big( \int_{[0,1]^2} C(\bm u) \diff C(\bm u) \Big)^2 \Bigg) \diff \xi.
\end{align*}
where $\bar C(\bm u) = 1 - u^{(1)}-u^{(2)} + C(\bm u)$ denotes the survival copula of $C$.
\end{proposition}

For the case $C = \Pi,$ where $\Pi$ denotes the independence copula, one can show that $\sigma^2_{\dbl} = 4/9, \sigma^2_{\sbl} = 32 (7/12 - 2\log(4/3))$ resulting in $\sigma^2_{\dbl}/\sigma^2_{\sbl} \approx 1.7428$, which has also been validated in a simulation experiment.

\section{Extensions to piecewise stationarity}
\label{sec:ext}

Environmental data typically involve different forms of non-stationarity. A particular source is seasonality, which may statistically be approached by restricting attention to seasons rather than years, bearing in mind that the inner-season variability should be approximately stationary. This idea may be approached mathematically by working with data satisfying the following assumption taken form \cite{BucZan23}.

\begin{condition}[Piecewise stationary observation scheme] \label{cond:s2}
For sample size $n\in\N$, we have observations $\bm{X}_{n,1}, \dots,  \bm{X}_{n,n}$ taking values in $\R^d.$ Moreover, 
for some block length sequence $(r_n)_n \subset \N$ diverging to infinity such that $r_n=o(n)$, we have 
\begin{multline*}
(\bm{X}_{n, 1}, \dots, \bm{X}_{n,n}) = (\bm{Y}_{1,1}, \dots, \bm{Y}_{1,r_n}, 
\bm{Y}_{2,1}, \dots, \bm{Y}_{2, r_n}, \dots \\
\dots, 
\bm{Y}_{\ndb, 1}, \dots, \bm{Y}_{\ndb, r_n}, \bm{Y}_{\ndb +1, 1}, \dots, \bm{Y}_{\ndb+1, n-\ndb r_n}),
\end{multline*}
where $\ndb =\lfloor n/r_n \rfloor$ and where $(\bm{Y}_{1,t})_t, (\bm{Y}_{2, t})_t, \dots$ denote i.i.d.\ copies from a stationary time series satisfying Condition~\ref{cond:mda} with continuous marginal c.d.f.\ $F$. Note that $\bm{Y}_{j,t}$ should be regarded as the $t$-th observation in the $j$-th season.
\end{condition}
We refer to \cite{BucZan23} for further discussions of Condition~\ref{cond:s2}, see in particular Remark 2.3.
For the rest of this section, we tacitly assume Condition \ref{cond:s2} and write $\bm{X}_j := \bm{X}_{n, j}$ for simplicity. Note that the triangular array $(\bm{X}_n)_n$ is $r_n$ dependent, which in fact simplifies the analysis of the disjoint block maxima method. For the sliding block maxima method however, mathematical challenges arise from the fact that the sliding block maxima sample is typically non-stationary. Indeed, for $x\in\R^d$, generally
\begin{align*}
\Pm\left( \bm{M}_{r,1} \leq x \right) 
&\neq \Pm \left( \bm{X}_2, \ldots, \bm{X}_r \leq \bm{x} \right) \cdot \Pm \left( \bm{X}_{r+1} \leq \bm{x} \right) 
= 
\Pm \left( \bm{M}_{r,2} \leq \bm{x} \right).
\end{align*}
In \cite{BucZan23}, Lemma 2.4, it is shown that this non-stationarity disappears asymptotically, which suggests that statistical methodology derived under stationarity assumptions (as in Section~\ref{sec:asy}) may also be applicable under Condition~\ref{cond:s2}. For deriving respective limit results, some modifications of the previous conditions are necessary.
First of all, the integrability conditions from Condition \ref{cond:h_int_std} take the following, slightly more involved form.

\begin{condition}\label{cond:h_int_S2}
There exists a $\nu > 2/\omega$ with $\omega$ from Condition \ref{cond:ser_dep} such that
\begin{compactenum}
\item[(a)] $\limsup_{r\to\infty} \sup_{1 \leq i \leq j \leq r} \int \int \abs{h(\bm{x}, \bm{y})}^{2+\nu} \on{d}\!\mathbb{P}_{\bm{Z}_{r,i}}(\bm{x}) \on{d}\!\mathbb{P}_{\bm{Z}_{r,j}}(\bm{y}) < \infty,$
\item[(b)] $\limsup_{r\to\infty} \sup_{1 \leq i \leq j \leq r} \Exp[\abs{h(\bm{Z}_{r,i},\bm{Z}_{r,j})}^{2+\nu}] < \infty$.
\end{compactenum}   
\end{condition}

It is worth noting that, if there exist monotone functions $g_1, g_2$ such that $|h(x,y)| \leq |g_1(x)| + |g_2(y)|,$ the inner supremum may be omitted; examples can be found in Section~\ref{sec:examples}.

Next, we quantify the average non-stationarity for the sliding block maxima. For $i,j \in \{1, \dots, r\}$, let 
\begin{align}\label{eq:thetaij}
\vartheta_{r,i,j} &:= \Exp[ h(\bm{Z}_{r,i}, \tilde{\bm{Z}}_{r,j}) ],
\qquad
\bar{\vartheta}_r := \frac{1}{r^2} \sum_{1 \leq i, j \leq r} \vartheta_{r,i,j},
\end{align}
where $(\tilde{\bm Z}_{r,j})_{j=1, \dots, r}$ is an independent copy of $(\bm Z_{r,j})_{j=1, \dots, r}$. Note that $\vartheta_{r,1,1} = \vartheta_r$ with $\vartheta_r$ from \eqref{eq:varthetar}, while $\vartheta_{r,i,j} \ne \vartheta_{r}$ in general. We do however have $\bar{\vartheta}_r = \vartheta_r + o(1)$ under the previous conditions (see also Lemma B.5 and B.6 in \citealp{BucZan23} for similar results).

\begin{lemma}\label{lem:unb_s2}
Suppose Conditions \ref{cond:ker_traf}, \ref{cond:s2} and \ref{cond:h_int_S2}(a),(b) are met and that $h$ is $\dlambda^{2d}$-a.e.\ continuous. Then, for $n \to \infty$, 
\[
\Exp[U_{n,r,Z}^{\sbl}] = \bar \vartheta_r + O(m^{-1}) , \quad 
\bar{\vartheta}_r =  \vartheta_r + o(1).
\]
\end{lemma}

This result suggests that the non-stationarity of the sliding block maxima method under Condition~\ref{cond:s2} may show up in the asymptotic bias of the U-statistic $U_{n,r}^{\sbl}$. The following assumption requires $r$ to be 
sufficiently large to make this bias negligible.

\begin{condition}[Negligibility of the bias due to non-stationarity]\label{cond:bias_S2}
The limit $D := \lim_{n \to \infty} D_n$ exists, where
\begin{equation}
D_n = \sqrt m\big (\bar{\vartheta}_r - \vartheta_r\big).
\end{equation}
\end{condition}

\begin{theorem}\label{thm:S2_U_conv} 
Within the setting of Condition \ref{cond:s2}, suppose that the block size and the underlying time series $(\bm Y_{j,t})_t$ satisfy Condition \ref{cond:ser_dep}(a),(b) and that the kernel satisfies Condition~\ref{cond:ker_traf}.  Additionally, for $\mbl=\dbl$, suppose that Condition~\ref{cond:h_int_std}(a) is met, and for $\mbl=\sbl$, suppose that Condition~\ref{cond:h_int_S2} and \ref{cond:bias_S2} are met.
Then, if $h$ is $\dlambda^{2d}$-a.e.\ continuous and bounded on compact sets, we have, for $n \to \infty$,
\begin{equation*}
\frac{\sqrt{m}}{f(\bm{a}_r, \bm{b}_r)}\left( \Unmb - \theta_r \right) \wconv 
\begin{cases}
\Norm{0, \sigma_{\dbl}^2}, & \mbl = \dbl \\
\Norm{D, \sigma_{\sbl}^2}, & \mbl = \sbl
\end{cases}
\end{equation*}
with $\sigma_{\mbl}^2$ from \eqref{eq:as_vars} satisfying $\sigma_{\sbl}^2 \le \sigma_{\dbl}^2.$ 
If, additionally, the limit $B=\lim_{n\to\infty}B_n$ with $B_n$ from \eqref{eq:bn} exists, then, again for $n \to \infty$,
\begin{equation*}
\frac{\sqrt{m}}{f(\bm{a}_r, \bm{b}_r)} \left( \Unmb  - \tilde{\vartheta}_r \right) \wconv 
\begin{cases}
\Norm{B, \sigma_{\dbl}^2}, & \mbl = \dbl, \\
\Norm{B + D, \sigma_{\sbl}^2}, & \mbl = \sbl
\end{cases}
\end{equation*}
with $\tilde \vartheta_r$ from \eqref{eq:varthetat}.
\end{theorem}

\section{Simulation study}
\label{sec:sim_study}

A Monte Carlo simulation study was conducted to evaluate the finite sample performance of two selected estimators based on U-statistics: the empirical variance (univariate) as well as Kendall's $\tau$ statistic (bivariate).
The study mainly aimed at comparing the disjoint and sliding block maxima method for various extreme value indices and time series models.

\subsection{Estimating the block maxima variance} \label{subsection:sim_study_var}

In Section~\ref{subsec:var}, the empirical variance based on sliding block maxima, $\hat{\sigma}^2_{n,r, \sbl}$, was found to be an asymptotically more efficient estimator of $\sigma_{r}^2 := \Var(M_{r,1})$ than its disjoint blocks counterpart, $\hat{\sigma}^2_{n,r, \dbl}$. We assess the performance in finite-sample situations for 
data-generating processes made up from the following marginal and temporal models:

\smallskip
\noindent
\textbf{Stationary distribution of $X_t$:} We consider the generalized Pareto distribution $\on{GPD}(0,1,\gamma)$ with shape parameter $\gamma \in \{ -0.4, -0.2, 0, 0.1\}$, see \eqref{eq:gpd}. Note that the largest value of $\gamma= 0.1$ is close to the non-integrability point 0.25 for the variance estimator. 

\smallskip
\noindent
\textbf{Time series models:} In addition to the i.i.d.\ case, two time series models were considered, each with three parameter choices. The first model is the (transformed) ARMAX(1) model, see Section \ref{subsec:var}, with time series parameter $\alpha \in \{0.25, 0.5, 0.75\}$; note that the extremal index is $\theta=1-\alpha$. As the second model we chose the Cauchy AR model (CAR), defined as the stationary solution $(Y_t)_t$ of the CAR recursion
\begin{equation*}\label{eq:car}
    Y_t = \phi Y_{t-1} + W_t, \quad t \in \Z, \quad (W_t)_t \stackrel{i.i.d.}{\sim} \on{Cauchy}(0,1),
\end{equation*}
with time series parameter $\phi \in \{0.25, 0.5, 0.75\}$. This corresponds to the extremal index $\theta = 1- \phi,$ see, e.g., Problem 7.9 in \cite{KulSou20}. Realizations from the model were transformed to the $\on{GPD}(0,1,\gamma)$ distribution by setting $X_t = F_{\gamma}^{\leftarrow}(F_Y(Y_t))$, where $F_Y$ and $F_\gamma$ denote the c.d.f.\ of the Cauchy(0,1) and the GPD(0,1,$\gamma$)-distribution, respectively.

\smallskip

Combining each marginal model with each time series models results in a total of $4 \times 7=28$ different models. 
Throughout, we chose to fix the block size to $r=90$, which roughly corresponds to the number of days in the summer months and which is a common block length in environmental applications.
The number of blocks, denoted as $m$, ranged from 10 to 100, resulting in corresponding sample sizes ranging from $n=900$ to $n=9,000$ observations. 
The performance of the estimators was assessed based on approximating the MSE, the squared bias and the variance of the estimators based on $N=10,000$ simulation repetitions. For assessing the bias, the true variance $\sigma_{90}^2$ was determined in a preliminary simulation experiment involving a huge sample of size $10^6$ drawn from the distribution of $M_{r,1}$; with one such sample for each of the 28 models. 

The results for the i.i.d.\ and the ARMAX-models are illustrated in Figure~\ref{fig:sim_study_var}, where we depict the ratio $\mathrm{MSE}(\sigma_{n,r,\dbl}^2) / \mathrm{MSE}(\sigma_{n,r,\sbl}^2)$ as a function of the number of seasons (results for the CAR-model are omitted because they are qualitatively similar). 
Across all considered numbers of seasons, tail indices  and time series parameter, the sliding blocks estimator consistently outperforms its disjoint blocks counterpart. Notably, the depicted ratio is significantly larger than one for small tail indices and for small sample sizes. This particular observation is promising because obtaining large sample sizes is sometimes challenging in the area of extreme value statistics. Also, it should be noted that the serial dependence does not substantially influence the relative performance (as was to be expected from the asymptotic results). Finally, we would like to report that the estimation variance was found to be of much larger order than the bias, whence the MSE-ratio  is nearly the same as the respective variance ratio.

\begin{figure}[t!]	
	\centering
	\makebox{\includegraphics[width=1\textwidth]{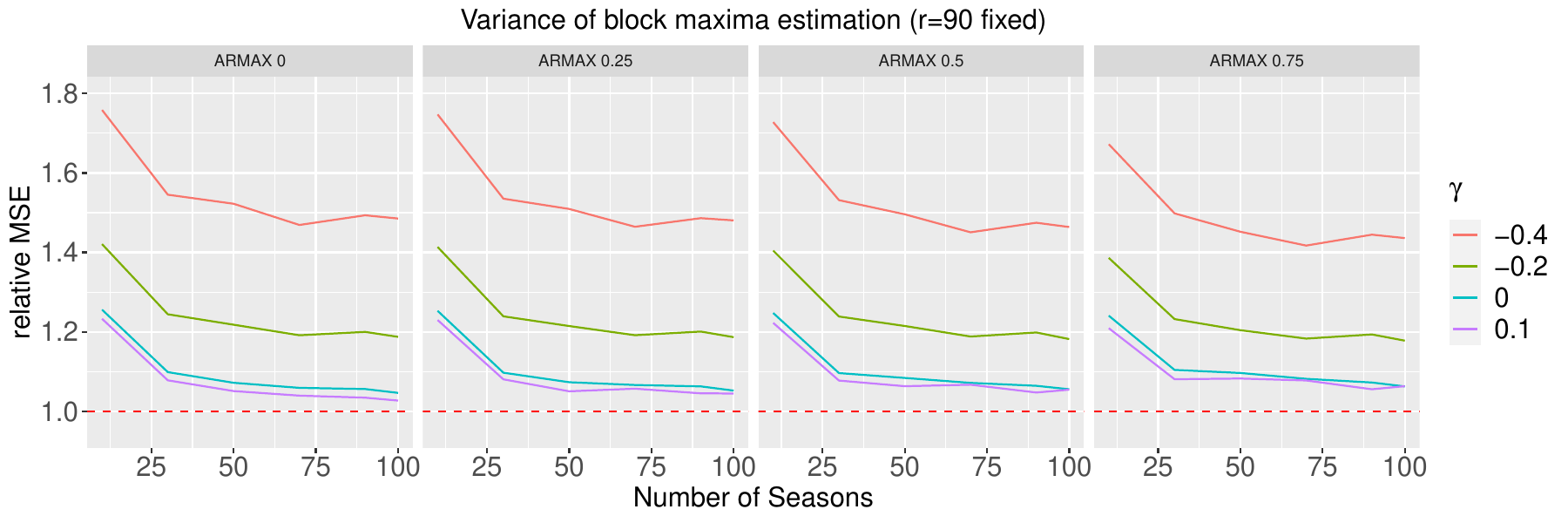}}\vspace{-.3cm}
	\caption{
For the estimation of $\sigma_r^2=\Var(M_{r,1})$, the ratio $\mathrm{MSE}(\sigma_{n,r,\dbl}^2) / \mathrm{MSE}(\sigma_{n,r,\sbl}^2)$ is depicted as a function of number of blocks $m$. 
}\label{fig:sim_study_var}
  \vspace{-.3cm}
\end{figure}

\subsection{Estimating Kendall's tau} \label{subsection:sim_study_tau}

We investigate the finite-sample performance in the bivariate case for the estimation of Kendall's $\tau= \tau_r = \tau(M_{r,1}^{(1)}, M_{r,1}^{(2)})$ based on the estimators $\hat{\tau}^{\dbl}_{n,r}$ and $\hat{\tau}^{\sbl}_{n,r}$ from~Section \ref{subsec:tau}. Note that both Kendall's $\tau$ and its estimators do not depend on the marginal distributions of $X_t^{(1)}$ and $X_t^{(2)}$ (in case they are continuous).
The data generating processes are as follows:

\smallskip
\noindent
\textbf{Time series models:} Three types of time series models were considered: bivariate versions of the ARMAX(1) and CAR(1) model from the previous section as well as i.i.d.\ observations. The bivariate ARMAX(1) model is defined as the stationary bivariate solution to the recursion equation:
\begin{equation*}
X_t^{(j)} = \max \{ \alpha X_{t-1}^{(j)}, (1-\alpha)W_t^{(j)}  \}, \quad t \in \Z, \quad j \in \{1,2\},
\end{equation*}
where $\alpha \in (0,1]$ and where $(\bm{W}_t)_t$ is an i.i.d.\ sequence with Fréchet(1)-distributed margins and with copula as specified below. Throughout, the value of $\alpha$ was fixed to 0.5; and the i.i.d.\ case is obtained for $\alpha=0$.
The bivariate CAR(1) model is defined as the stationary solution of the bivariate CAR(1) recursion
\begin{equation*}
    X_t^{(j)} = \phi X_{t-1}^{(j)} + W_t^{(j)}, \quad t \in \Z, \quad j \in \{1,2\},
\end{equation*}
where $\phi \in (0,1]$ and  where $(\bm{W}_t)_t$ is an i.i.d.\ sequence with Cauchy(1) margins and with copula as specified below.  Throughout, the value of $\phi$ was fixed to 0.5.

\smallskip
\noindent
\textbf{Copula of $\bm W_t$:} Seven different copulas were considered: the independence copula, the  Gaussian copula, the $t_\nu$-copula with $\nu=4$ degrees of freedom, and the Gumbel-Hougard copula, where the parameter of three last-named copulas was chosen in such a way that the associated value of Kendall's tau is in $\{0.3, 0.6\}$. Note that the Gaussian copula is tail independent, while the $t$- and Gumbel copula exhibit upper tail dependence. 
The upper tail dependence coefficients as a function of Kendall's tau are given by $2 \cdot t_{5}(-\sqrt{5(1-\sin(\pi \tau /2))/(1+ \sin(\pi \tau /2)}) \in \{0.23, 0.5\}$ and $2-2^{1-\tau} \in \{0.375, 0.68\}$ for the $t_4$ and Gumbel-Hougard copula, respectively; 
see \cite{Emb01}.

\smallskip

Overall, we obtain $3 \times 7=21$ different models. As in the previous section, we fix the block length to $r=90$ and vary $m$ between 10 and 100, resulting in sample sizes $n=mr$ ranging from 900 to 9,000 observations. 
The estimators are evaluated in terms of the mean squared error (MSE), the bias and the variance, based on $N=1,000$ simulation repetitions. The true value of $\tau_r$ was assessed in a preliminary simulation involving a sample of size $100,000$ from $\bm M_{r,1}$.

The results are presented in Figure \ref{fig:sim_study_ktau}, where we restrict attention to the CAR(1) model, as the performance in the other two time series is nearly identical. As in the previous section, the bias was found to be of much smaller order than the variance, whence we further restrict attention to  $\mathrm{MSE}(\tau_{n,r}^{\dbl})$ and to the MSE ratio $\mathrm{MSE}(\tau_{n,r}^{\dbl}) / \mathrm{MSE}(\sigma_{n,r}^{\sbl})$.
We observe that the sliding blocks estimator consistently outperforms the disjoint blocks counterpart. The level of dependence impacts the performance in that the estimation is more precise for higher dependence (for both estimators), and in that the advantage of the sliding blocks estimator over its disjoint blocks counterpart is highest for low levels of dependence/independence. Furthermore, as in the previous section, the sliding blocks estimator's advantage is slowly decreasing in the number of blocks.

\begin{figure}[t!]
	\centering
	\makebox{\includegraphics[width=1\textwidth]{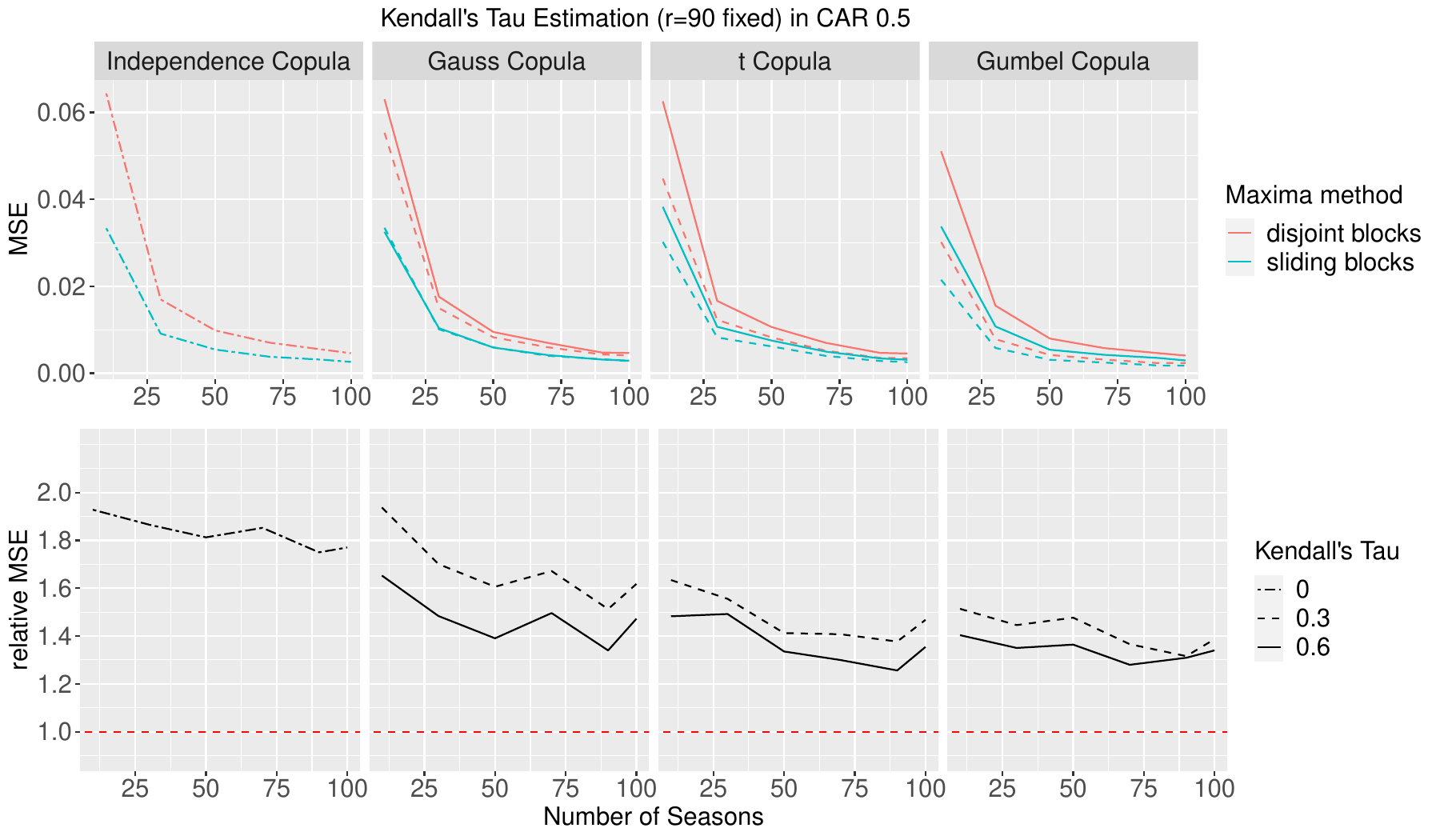}}\vspace{-.3cm}
	\caption{
MSE of $\hat \tau_{n,r}^{\mbl}$ (upper panel) and MSE ratio $\mathrm{MSE}(\tau_{n,r}^{\dbl}) / \mathrm{MSE}(\sigma_{n,r}^{\sbl})$ (lower panel) 
plotted against the number of blocks $m$.
} \label{fig:sim_study_ktau}
  \vspace{-.3cm}
\end{figure}

\section{Proofs}
\label{sec:proofs}

Recall the definitions of $\theta_r$, $\vartheta_r$, $\sigma^2_{\mbl}$ from \eqref{eq:thetar}, \eqref{eq:varthetar} and (\ref{eq:as_vars}), respectively. 
\subsection{Proofs for Section \ref{sec:asy}}
\begin{proof}[Proof of Theorem \ref{thm:U_conv}]
Recall the definition of $U_{n,r,\bm Z}^{(\mbl)}$  in \eqref{eq:u_stat_resc}. By Condition \ref{cond:ker_traf} we have 
\[
\frac{\Unmb - \theta_r}{f(\bm{a}_r, \bm{b}_r)}
 = U_{n,r,\bm Z}^{(\mbl)} - \vartheta_r. 
\]
Hence it suffices to show that
\begin{align}
&\sqrt{m} \cdot ( \tilUnmb - \vartheta_r) \wconv 
 \Norm{0, \sigma_{\mbl}^2},\label{eq:proof_u_conv_mbl} 
 \end{align}
for $\mbl \in \{\dbl, \sbl\}$ and 
\begin{align}
&\sigma_{\sbl}^2 \leq \sigma_{\dbl}^2 \label{eq:proof_sbl_better}.
\end{align}
For the proof of \eqref{eq:proof_u_conv_mbl} we will use a Hoeffding decomposition and verify weak convergence of the linear part to the normal limit and $L^2$-convergence to zero of the asymptotically degenerate part. For both parts we will employ common blocking techniques to deal with the serial dependence, (see e.g., \cite{DehPhi02}, page 31). Define 
\begin{align}
& h_{1,r} \label{eq:h1r}
\colon \R^{d} \to \R, \quad \bm x \mapsto h_{1,r}(\bm x) := \Exp[h(x,\bm{Z}_{r,1})] - \vartheta_r    \\
& h_{2,r} \nonumber
\colon \R^{2d} \to \R \quad (\bm x, \bm y) \mapsto h_{2,r}(\bm x,\bm y) := h(\bm x,\bm y) - h_{1,r}(\bm x) - h_{1,r}(\bm y) - \vartheta_r
\end{align}
and notice the algebraic identity
\begin{align}\label{eq:proof_U_Zerl}
\Unmb - \vartheta_r 
&=  \nonumber
\frac{2}{\nmb} \sum_{i \in \Imb} h_{1,r}(\bm{Z}_{r,i}) + \frac{2}{\nmb \cdot(\nmb-1)} \sum_{(i,j) \in \Jmb} h_{2,r}(\bm{Z}_{r,i}, \bm{Z}_{r,j}) \\
&\equiv 
L_{n,r}^{\mbl}
+ 
D_{n,r}^{\mbl}.
\end{align}

\medskip
\noindent
\textbf{Proof of (\ref{eq:proof_u_conv_mbl}) for mb $\mathbf =$ db:} we start by proving
\begin{equation}\label{eq:proof_u_conv_lin_part}
\sqrt m \cdot L_{n,r}^{\dbl}
=
\frac{2}{\sqrt{m}} \sum_{i \in \Idb} h_{1,r}(\bm{Z}_{r,i}) 
\wconv \Norm{0, \sigma^2_{\dbl}}.    
\end{equation} 
By Lemma \ref{lem:switch_to_indep},  we may switch to i.i.d.\ copies of $\bm{Z}_{r,i}$, denoted by $\tilde{\bm{Z}}_{r,i}$. Hence, Ljapunov's central limit theorem becomes applicable. Consider first the case $\Var(h_{1}(\bm{Z})) > 0.$ In view of the fact $h_{1,r}(\tilde{\bm{Z}}_{r,i})$ is centered, we need to check the Ljapunov condition: 
\[
\exists \delta > 0\colon \quad 
\lim_{n \to \infty} \frac{\sum_{i \in \Idb} \Exp\big[\abs{h_{1,r}(\tilde{\bm{Z}}_{r,i})}^{2+\delta}\big]}
{\left( \sum_{i \in \Idb} \Exp[h_{1,r}(\tilde{\bm{Z}}_{r,i})^2] \right)^{1+ \delta/2} }
= 0.
\]
Take $\delta = \nu$ with $\nu$ from Condition \ref{cond:h_int_std}.  By stationarity, Jensen's inequality and Con\-di\-tion~\ref{cond:h_int_std}~(a) we obtain that
$ \sum_{i \in \Idb} \Exp[\abs{h_{1,r}(\tilde{\bm{Z}})}^{2+\nu}] = O(m)$. 
Hence, using Lemma \ref{lem:h_1,r_conv} and stationarity, we get 
\begin{align*}
 \frac{\sum_{i \in \Idb} \Exp\big[\abs{h_{1,r}(\tilde{\bm{Z}}_{r,i})}^{2+\delta}\big]}
{\left( \sum_{i \in \Idb} \Exp[h_{1,r}(\tilde{\bm{Z}}_{r,i})^2] \right)^{1+ \delta/2} }
= \frac{O(m^{-\nu/2})}{\Var(h_{1,r}(\bm{Z}_{r,1}))^{1+\nu/2}} =o(1),
\end{align*}
where we used the assumption $\Var(h_{1}(\bm{Z})) > 0.$ Lemma \ref{lem:h_1,r_conv} then implies the asymptotic normality in (\ref{eq:proof_u_conv_lin_part}).

Consider now the case $\Var(h_1(\bm{Z})) = 0.$ By stationarity and independence we obtain 
$\Var(m^{-1/2} \sum_{i \in \Idb} h_{1,r}(\bm{\tilde{Z}}_{r,1})) = \Var(h_{1,r}(\bm{Z}_{r,1})) = o(1),$ where we used Lemma \ref{lem:h_1,r_conv} in the last equality. Since $h_{1,r}(\bm{Z}_{r,i})$ is centered, the weak convergence to $\Norm{0,0}$ and thus \eqref{eq:proof_u_conv_lin_part} follow.

In the next part we show that the (asymptotically) degenerate part converges to zero in $L^2$, i.e., $\Exp[(\sqrt m \cdot L_{n,r}^{\dbl})^2]=o(1)$. For that purpose, it is sufficient to show that
\begin{align}\label{eq:proof_u_conv_deg_part}
\frac{m}{\ndb^4} \sum_{\genfrac{}{}{0pt}{}{(i_1,  i_2) \in \Jdb}{(j_1, j_2) \in \Jdb}} \Exp[h_{2,r}(\bm{Z}_{r,i_1},\bm{Z}_{r,i_2}) h_{2,r}(\bm{Z}_{r,j_1},\bm{Z}_{r,j_2})]
=o(1).
\end{align}
The Cauchy-Schwarz inequality, standard inequalities for the expectation and Condition \ref{cond:h_int_std} imply that 
\begin{align*}
\sup_{\genfrac{}{}{0pt}{}{(i_1,  i_2) \in \Jdb}{(j_1, j_2) \in \Jdb}}  \big| \Exp[h_{2,r}(\bm{Z}_{r,i_1},\bm{Z}_{r,i_2}) h_{2,r}(\bm{Z}_{r,j_1},\bm{Z}_{r,j_2})] \big| 
&\le 
\sup_{s \in  \N} \Exp[h_{2,r}^2(\bm Z_{r,1}, \bm Z_{r,1+s})]  \\
&= O(1).
\end{align*}
Consider a tuple $(\bm i, \bm j) = (i_1, i_2, j_1, j_2) \in \{(i_1,i_2)\in \Jdb, (j_1, j_2) \in \Jdb\}$ such that both the distance between the smallest, $\min(\bm i, \bm j)$, and the second smallest index and the largest, $\max(\bm i, \bm j)$, and the second largest index is at most $2r$. Clearly, the cardinality of the set of all those $(\bm i, \bm j)$ is of the order $O(m^2)$, whence the expression in \eqref{eq:proof_u_conv_deg_part} with the sum restricted to those tuples is of the order $O(m^{-1})$. It is hence sufficient to consider the sum over those summands for which either the distance between the smallest index and all other indices is strictly larger than $2r$, or the distance between the largest index and all other indices is strictly larger than $2r$. We only consider the first case, as the other can be treated similarly. Without loss of generality, let $i_1$ be the smallest index, and let $\JJdb=$ denote the respective set of indices, that is, $\JJdb=\{(\bm i, \bm j) \in \Jdb \times \Jdb: i_2-i_1> 2r, j_1 -i_1> 2r\}$.

For each tuple $(\bm i, \bm j) \in \JJdb$, we may use Berbee's coupling Lemma \citep{berbee} to construct a random variable $\bm{Z}_{r,i_1}^\ast$ having the same distribution as $\bm{Z}_{r,i_1}$ that is independent of $(\bm{Z}_{r,i_2}, \bm{Z}_{r,j_1}, \bm{Z}_{r,j_2})$ and which satisfies $\mathbb P ( \bm{Z}_{r,i_1} \neq \bm{Z}_{r,i_1}^\ast ) \leq \beta(\sigma(\bm{Z}_{r,i_1}), \sigma(\bm{Z}_{r,i_2}, \bm{Z}_{r,j_1}, \bm{Z}_{r,j_2})) \leq \beta(r)$ where  $\sigma(X)$ denotes the initial $\sigma$-field of $X.$ Now decompose 
\begin{align*}
 &\phantom{{}={}} \Exp[h_{2,r}(\bm{Z}_{r,i_1},\bm{Z}_{r,i_2}) h_{2,r}(\bm{Z}_{r,j_1},\bm{Z}_{r,j_2})]  \\
&= \Exp[\left \{  h_{2,r}(\bm{Z}_{r,i_1},\bm{Z}_{r,i_2}) - h_{2,r}(\bm{Z}^\ast_{r,i_1},\bm{Z}_{r,i_2})  \right\} h_{2,r} (\bm{Z}_{r,j_1},\bm{Z}_{r,j_2})] \\
& \hspace{2cm} + \Exp[h_{2,r}(\bm{Z}^\ast_{r,i_1},\bm{Z}_{r,i_2}) h_{2,r} (\bm{Z}_{r,j_1},\bm{Z}_{r,j_2})] =: I^{(\bm{i}, \bm{j})}_{1,n} + I^{(\bm{i}, \bm{j})}_{2,n}.
\end{align*}
Using stationarity, basic properties of the conditional expectation and the properties of $\bm{Z}^\ast_{r,i_1}$ we obtain, via conditioning on $(\bm{Z}_{r,i_2}, \bm{Z}_{r,j_1}, \bm{Z}_{r,j_2})$, that $I^{(\bm{i}, \bm{j})}_{2,n} \equiv 0.$ 

Next, repeated applications of Hölder's inequality imply that, uniformly in $(\bm i, \bm j) \in \JJdb$,
\begin{align*}
\abs{I^{(\bm{i}, \bm{j})}_{1,n}} 
&\leq \Exp\big[\ind{\bm{Z}_{r,i_1} \neq \bm{Z}_{r,i_1}^\ast} \big| \big \{ h_{2,r}(\bm{Z}_{r,i_1},\bm{Z}_{r,i_2}) - h_{2,r}(\bm{Z}^\ast_{r,i_1},\bm{Z}_{r,i_2})  \big\}  \\ & \hspace{7cm} \times h_{2,r} (\bm{Z}_{r,j_1},\bm{Z}_{r,j_2})\big|\big] \\
&\lesssim  \beta(r)^{\nu/(2+\nu)}, 
\end{align*}
where we have used Condition \ref{cond:h_int_std}. Overall,
\begin{align*}
\frac{m}{\ndb^4} \sum_{(\bm i, \bm j) \in \JJdb} \abs{\Exp[h_{2,r}(\bm{Z}_{r,i_1},\bm{Z}_{r,i_2}) h_{2,r}(\bm{Z}_{r,j_1},\bm{Z}_{r,j_2})]} 
&\lesssim  
m \cdot  \sup_{(\bm i, \bm j) \in I} \abs{I^{(\bm{i}, \bm{j})}_{1,n} + I^{(\bm{i}, \bm{j})}_{2,n}} \\
&\lesssim 
\Big( m^{1+2/\nu} \beta(r) \Big)^{\nu/(2+\nu)}
\end{align*}
which converges to zero by 
Condition \ref{cond:ser_dep} (c) as $2/\nu < \omega$. This implies (\ref{eq:proof_u_conv_deg_part}), and in combination with \eqref{eq:proof_U_Zerl} and (\ref{eq:proof_u_conv_lin_part}) we obtain (\ref{eq:proof_u_conv_mbl}). 

\medskip
\noindent
\textbf{Proof of (\ref{eq:proof_u_conv_mbl}) for mb $\mathbf =$ sb:} 
In order to show that the degenerate part of the rescaled sliding blocks U-statistic converges to zero, it is sufficient to show that
\[
\frac{m}{\nsb^4} \sum_{\genfrac{}{}{0pt}{}{(i_1,  i_2) \in \Jsb}{(j_1, j_2) \in \Jsb}} \Exp[h_{2,r}(\bm{Z}_{r,i_1},\bm{Z}_{r,i_2}) h_{2,r}(\bm{Z}_{r,j_1},\bm{Z}_{r,j_2})]
\to 0.
\]
This can be worked out analogously to the disjoint case: again, we may restrict the sum in the upper display to tuples in $\Jsb = \{(\bm i, \bm j) \in \Jsb \times \Jsb: i_2-i_1> 2r, j_1 -i_1> 2r\},$ as the set of the remaining tuples is of the order $O((nr)^2)$. We can then copy the disjoint blocks proof verbatim by replacing $\JJdb$ and $\ndb$ with $\JJsb$ and $\nsb$.

It remains to show 
\begin{equation*}
\frac{2\sqrt{m}}{n} \sum_{i \in \Isb} h_{1,r}(\bm{Z}_{r,i}) 
\wconv \Norm{0, \sigma^2_{\sbl}}.  
\end{equation*}
For this purpose use Theorem \ref{thm:slid_gws} with $f_{r,s} := h_{1,r}, \, f := h_1$ and note that all conditions are satisfied,  where we use Lemma \ref{lem:overlap_wconv} and an easy adaptation of Lemma B.15 in \cite{BucZan23} to obtain the weak convergence condition in \eqref{eq:wl2}.

\medskip
\noindent
\textbf{Proof of (\ref{eq:proof_sbl_better}):} 
The inequality follows from Lemma A.10 in \cite{ZouVolBuc21}, where $X_{n,i} := h_{1,r}(\bm{Z}_{r,i})$ and the preconditions of Lemma A.10 can be deduced from Condition \ref{cond:ser_dep}(a), (c) and \ref{cond:h_int_std}(a). 
\end{proof}

\begin{proof}[Proof of Corollary \ref{cor:U_conv_bias}]
By Condition \ref{cond:ker_traf} and the assumption on $B_n$, we have 
\begin{align*} 
\frac{\sqrt{m}}{f(\bm{a}_r, \bm{b}_r)} \big( \theta_r - f(\bm{a}_r, \bm{b}_r) ( \vartheta + \ell(\bm{a}_r, \bm{b}_r) \big) 
= \sqrt{m} ( \vartheta_r - \vartheta) = B_n = B + o(1).
\end{align*}
Hence, the assertion follows from  Theorem \ref{thm:U_conv} and Slutsky's theorem.
\end{proof}

\subsection{Proofs for Section \ref{sec:examples}}
\label{sec:proofs_ex}

\begin{proof}[Proof of Corollary \ref{cor:var_est}]
Note that 
$|h_{\on{Var}}(x,y)|^{2+\nu/2} \leq 2^{2+\nu/2}(|x|^{4+\nu}+|y|^{4+\nu})$. Hence, by assumption, Condition \ref{cond:h_int_std} is met for a $\nu > 0$ as in the formulation of Theorem \ref{thm:U_conv}. The statement follows by the continuity of $h_{\on{Var}}$ and Example \ref{ex:kernels}.
\end{proof}

\begin{proof}[Proof of Equation \eqref{eq:var-armax}]
Fix $\gamma < 1/4$ and omit the lower index 1 everywhere; e.g., write $Z_{r}$ instead of $Z_{r,1}$. We need to verify the conditions of Corollary \ref{cor:var_est}. 

We start by proving Assumption \ref{cond:mda},  for which we restrict attention to the case $\gamma > 0$ since the other cases can be treated similarly. Using $F_W^{\leftarrow}(F_\gamma(t)) = -1/\log \{ 1-(1+\gamma t)^{-1/\gamma} \}$ for $t > 0 $ and equation (10.5) from \cite{BeiGoeSegTeu04} we have 
\begin{align*}
\Pm\left( \frac{M_{r}-b_r}{a_r}  \le x \right)
= \,&\exp \left[ - \frac{1+(1-\alpha)(r-1)}{F_W^{\leftarrow}(F_\gamma(a_rx+b_r))} \right] \\
= \,& \exp\left[r(1-\alpha) \log\{ 1-(1+\gamma(a_rx+b_r))^{-1/\gamma} \} \right]+o(1) \\
= \,& \exp\left[ -r(1-\alpha) (1+\gamma(a_rx+b_r))^{-1/\gamma} \right] + o(1) \\
= \,& G_\gamma(x) +o(1),
\end{align*}
where we substituted $a_r = (r(1-\alpha))^\gamma, b_r= \lbrace (r(1-\alpha))^\gamma -1 \rbrace /\gamma.$ Hence Assumption \ref{cond:mda} is satisfied. 

Condition~\ref{cond:ser_dep}(a) holds by assumption. Conditions (b) and (c) hold since $r = o(n^3)$ and since there exists $c>0$ with $\alpha(k) \leq \beta(k) \leq \exp(-ck)$ for $k \in \N$ by the discussion in Section~\ref{subsec:var}. Note that (c) does hold for all $\omega > 0.$

In view of the latter statement, 
it remains to prove 
$\limsup_r \Exp[|Z_{r}|^{4+\nu}] < \infty$ for some $\nu > 0$. Note that this in turn is equivalent to $\limsup_r \Exp[|Z'_{r}|^{4+\nu}] < \infty$, where $Z'_{r} := (M_r - {b}'_r)/{a}'_r$ and where ${b}'_r \in \R , {a}'_r > 0$ are sequences with ${Z}'_{r} \rightsquigarrow \GEV{\mu, \sigma, \gamma}$ for some $\mu \in \R, \sigma >0$.

Define ${a}'_r := r^\gamma$ and ${b}'_r := (r^\gamma-1)/\gamma$, where the latter is defined by continuity as ${b}'_r= \log r$ if $\gamma = 0.$ The p.d.f.\ of ${Z}'_r$ is then given by 
\[
f_{{Z}'_r}(t) = (1-\alpha + \alpha/r) \cdot 
\begin{cases}
 (1+\gamma t)^{-(1+1/\gamma)} \left( 1- \frac{(1+t\gamma)^{-1/\gamma}}{r} \right)^{r(1-\alpha)+\alpha-1}, & \gamma \neq 0 \\
 e^{-t} \left( 1-\frac{e^{-t}}{r} \right) ^{r(1-\alpha)+\alpha -1}, & \gamma = 0
\end{cases}
\]
for $t \in \on{supp}(\tilde{Z}_r).$ We will only present the case $\gamma > 0$ as the other cases use similar ideas. Substituting $1+t\gamma$, we obtain
\begin{align*}
&\Exp[|\tilde{Z}_r|^{4+\nu}] \\
=&\, \frac{1-\alpha + \alpha/r}{\gamma} \int_{1/r^\gamma}^\infty \Big(\frac{|t-1|}{\gamma}\Big)^{4+\nu} t^{-1-1/\gamma} \Big(1-\frac{t^{-1/\gamma}}{r}\Big)^{r(1-\alpha)+\alpha-1} \on{d}\!t \\ 
\leq &\, \frac{1-\alpha + \alpha/r}{\gamma^{5+\nu}} \Bigg \lbrace \int_{1/r^\gamma}^{1/2} t^{-1-1/\gamma} \Big(1-\frac{t^{-1/\gamma}}{r}\Big)^{r(1-\alpha)+\alpha-1} \on{d}\!t   \\ 
& \hspace{3cm} +\int_{1/2}^\infty t^{3+\nu - 1/\gamma} \Big(1-\frac{t^{-1/\gamma}}{r}\Big)^{r(1-\alpha)+\alpha-1} \on{d}\!t \Bigg \rbrace 
=: I_{r1} + I_{r2}
\end{align*}
By the monotone convergence theorem the first integral converges to $\int_0^{1/2} t^{-1-1/\gamma} \exp(-(1-\alpha)t^{-1/\gamma}) \on{d}\!t < \infty$; hence $\lim_{r\to\infty} I_{r1} < \infty.$  
Finally, let $\nu = 1/(2\gamma)-2$ and invoke the monotone convergence theorem again to obtain 
\[
\lim_{r\to\infty} I_{r2}
= \frac{1-\alpha}{\gamma^{3+\gamma/2}} \int_{1/2}^\infty t^{1-1/(2\gamma)} \exp(-(1-\alpha)t^{-1/\gamma}) \on{d}\!t < \infty
\]
as $1-1/(2\gamma) < -1$. Overall, we have shown that $\limsup_r \Exp[|Z_{r}|^{4+\nu}] < \infty$ as asserted.

Using similar ideas as before, one can show that $n = o(r^3)$ implies $\lim_{n\to\infty} B_n = 0.$
\end{proof}

\begin{proof}[Proof of Proposition \ref{ex:pwm}]
Write $h_{\on{pwm},2}=h_{\on{pwm}}$ and $\tilde h_{\on{pwm},2}=\tilde h_{\on{pwm}}$.
First of all, we have
\[
\frac{\sqrt{m}}{a_r}\big\{  U_{n,r}^{\mbl}(h_{\on{pwm}}) -  U_{n,r}^{\mbl}(\tilde{h}_{\on{pwm}})  \big\} \\
= 
S_n^{\mbl} +R_n^{\mbl} 
\]
where
\begin{align*}
S_n^{\mbl} 
&= \sqrt{m} \binom{\nmb}{2}^{-1} \sum_{\genfrac{}{}{0pt}{}{(i,j) \in \Jmb}{j-i > 2r}} \big\{ h_{\on{pwm}}(Z_{r,i}, Z_{r,j}) - \tilde{h}_{\on{pwm}}(Z_{r,i}, Z_{r,j})\big\}  \\
R_n^{\mbl} 
&= \sqrt{m}  \binom{\nmb}{2}^{-1} \sum_{\genfrac{}{}{0pt}{}{(i,j) \in \Jmb}{j-i \le 2r}} \big\{ h_{\on{pwm}}(Z_{r,i}, Z_{r,j}) - \tilde{h}_{\on{pwm}}(Z_{r,i}, Z_{r,j})\big\}.
\end{align*}
The number of summands in $R_n^{\mbl}$ is of the order $O(nr)$ for $\mbl=\sbl$ and of the order $O(m)$ for $\mbl=\dbl$, whence $R_n^{\mbl}=O_{L^2}(m^{-1/2})=o_{L^2}(1)$ by the integrability assumption.

Next, we have
\begin{align*}
S_n^{\mbl} & = \sqrt{m}\binom{\nmb}{2}^{-1}  \sum_{\genfrac{}{}{0pt}{}{(i,j) \in \Jmb}{j-i > 2r}} \frac{1}{2}\mathbf{1}(Z_{r,i} = Z_{r,j}) Z_{r,i}
\end{align*}
which is zero with probability one by the no-ties assumption; note that all indices in the sum refer to blocks that do not overlap. 

The second statement follows from Corollary \ref{cor:U_conv_bias}, applied to $U_{n,r}^{\mbl}(\tilde h_{\on{pwm}})$. Finally, the inequality for the asymptotic variances can be found in \cite{BucZan23}.
\end{proof}

\begin{proof}[Proof of Proposition~\ref{ex:kend_tau}]
Recall Example \ref{ex:kernels}(5) and apply Theorem \ref{thm:U_conv}. A short calculation yields the formulas for the asymptotic variances.
\end{proof}

\subsection{Proofs for Section \ref{sec:ext}}

\begin{proof}[Proof of Lemma \ref{lem:unb_s2}]
For $\xi \in (0,1)$, let $\xi_r=1+\lfloor r\xi \rfloor$. Then,
\begin{align*}
\bar \vartheta_r 
=
\int_0^1 \int_0^1 
\Exp[ h(\bm Z_{r,\xi_r}, \tilde{\bm Z}_{r,\xi'_r}) ]
\diff \xi \diff \xi'
\end{align*}
By Lemma~\ref{lem:as_stat_Zrxi}, we have $\bm Z_{r,\xi_r} \wconv \bm Z \sim G$ for any $\xi \geq 0$. Hence, by independence and the continuous mapping theorem, $h(\bm Z_{r,\xi_r}, \tilde{\bm Z}_{r,\xi_r'}) \wconv h(\bm Z, \bm{\tilde{Z}})$. Therefore, by the previous display, dominated convergence (use Condition~\ref{cond:h_int_S2}(a)) and Example 2.21 in \cite{Van98}, we obtain $\bar \vartheta_r=\vartheta+o(1)$. This implies the second statement, since $\vartheta_r = \vartheta + o(1)$ by Lemma~\ref{lem:varthetar}.

For the first convergence assume that $n=mr$ for simplicity. We have
\begin{align*}
\Exp[U_{n,r,Z}^{\sbl}] 
&=
\frac{1}{\nsb(\nsb -1)} \sum_{\genfrac{}{}{0pt}{}{1 \leq i\neq j \leq \nsb}{j - i > 2r}} \Exp[h(\bm Z_{r,i}, \bm Z_{r,j})] + O(m^{-1}),
\end{align*}
where the $O$-term is due to leaving nearby summands out.
Next, by independence, piecewise stationarity and including nearby summands again,
\[
\sum_{\genfrac{}{}{0pt}{}{1 \leq i\neq j \leq \nsb}{j - i > 2r}} \Exp[h(\bm Z_{r,i}, \bm Z_{r,j})]
=
\sum_{\genfrac{}{}{0pt}{}{1 \leq i\neq j \leq \nsb}{j - i > 2r}} \Exp[h(\bm Z_{r,i}, \tilde{\bm Z}_{r,j})]
= m^2 \sum_{1 \leq i, j \leq r} \vartheta_{r,i,j} + O(rn).
\]
Overall,
\begin{align*}
\Exp[U_{n,r,Z}^{\sbl}] 
= 
\frac{m^2r^2}{\nsb(\nsb -1)} \bar{\vartheta}_r + O(m^{-1})
= \bar \vartheta_r + O(m^{-1}).
\end{align*}
\end{proof}

\begin{proof}[Proof of Theorem \ref{thm:S2_U_conv}]
For $\mbl=\dbl$, $(\bm Z_{r,i})_{i \in \Idb}$ is an i.i.d. sample. Thus the proof essentially is an easier version of the proof of Theorem~\ref{thm:U_conv}.

For $\mbl = \sbl$ note that, by Lemma~\ref{lem:unb_s2}, Conditions~\ref{cond:bias_S2} and \ref{cond:ker_traf}, it is sufficient to show that $\sqrt{m}(U_{n,r,Z}^{\sbl} - \Exp[U_{n,r,Z}^{\sbl}]) \wconv \Norm{0, \sigma^2_{\sbl}}.$
Note that we might replace $n-r+1$ with $n,$ since $(n-r+1)/n = 1 +O(m^{-1}).$
Unlike in the situation from Theorem~\ref{thm:U_conv}, the sliding block maxima sample is not stationary anymore, which requires a different version of the Hoeffding decomposition. For $1 \le i,j \le n$, introduce functions $h_{1,r,i} \colon \R^{d} \to \R$ and $h_{2,r,i,j} \colon \R^{2d} \to \R$ by
\begin{align*}
h_{1,r,i}(\bm x) &:= \Exp[ h \left(\bm x, \bm{Z}_{r,i} \right) ]- \vartheta_{r,i,i} \\
h_{2,r,i,j}(\bm x, \bm y) &:= h(\bm x, \bm y) - h_{1,r,i}(\bm x) - h_{1,r,j}(\bm y) - \vartheta_{r,i,j},
\end{align*}
with $\vartheta_{r,i,j}$ from \eqref{eq:thetaij}.
Note, that by Lemma \ref{lem:unb_s2}
\begin{align}\label{eq:proof_U_Zerl_S2}
&\phantom{{}={}} U_{n,r,Z}^{\sbl} - \frac{2}{n(n-1)} \sum_{1 \leq i < j \leq n}  \Exp[h(\bm{Z}_{r,i}, \bm{Z}_{r,j})] \nonumber \\ 
& = \frac{2}{n} \sum_{i=1}^n h_{1,r,i}(\bm{Z}_{r,i}) + \frac{2}{n(n-1)} \sum_{1 \leq i < j \leq n} h_{2,r,i,j}(\bm{Z}_{r,i}, \bm{Z}_{r,j}) + O(m^{-1})\\
\nonumber &\equiv 
L_{n,r}^{\sbl}
+ 
D_{n,r}^{\sbl} + O(m^{-1})
\end{align}
The asymptotic normality of $ \sqrt{m}L_{n,r}^{\sbl}$ follows from Theorem \ref{thm:slid_gws-ns} with $f_{r,i} := 2 h_{1,r,i},$ where the preconditions are met since the time series is piecewise stationary and by assumption; moreover,  \eqref{eq:wl2} is a consequence of Lemma \ref{lem:overlap_h1_wconv_S2}. We omit the proof of $\sqrt{m}D_{n,r}^{\sbl} = o_\Pm(1)$ as the proof is similar to the proof of the respective statement in the proof of Theorem \ref{thm:U_conv}.
\end{proof}

\section{Auxiliary results}

\subsection{Disjoint blocks - stationary case}

\begin{lemma}[Convergence of $\vartheta_r$] \label{lem:varthetar}
Suppose that $h$ is $\dlambda^{2d}$-a.e.\ continuous and that there exists $\nu > 0$ with $\limsup_{r\to\infty} \int \int |h(\bm x,\bm y)|^{1+\nu} \on{d}\!\mathbb{P}_{\bm{Z}_{r,1}}(\bm{x}) \on{d}\!\mathbb{P}_{\bm{Z}_{r,1}}(\bm{y}) < \infty.$ 
Then $\lim_{r \to \infty} \vartheta_r=\vartheta$ with $\vartheta_r$ and $\vartheta$ from \eqref{eq:varthetar} and \eqref{eq:vartheta}, respectively. 
\end{lemma}

\begin{proof}
We have $h(\bm{Z}_{r,1}, \bm{\tilde{Z}}_{r,1}) \wconv h(\bm Z, \bm{\tilde{Z}})$ by independence and the continuous mapping theorem. The assertion then follows from Example 2.21 in \cite{Van98} and the integrability assumption.
\end{proof}

Recall the definition of $h_{1,r}$ from \eqref{eq:h1r}.

\begin{lemma}[Weak convergence of $h_{1,r}(\bm Z_{r,1})$]\label{lem:h_1,r_conv}
Suppose that Condition \ref{cond:h_int_std}(a) holds and that $h$ is $\dlambda^{2d}$-a.e.\ continuous and bounded on compact sets. Then, for $r \to \infty$,
\begin{equation*}
h_{1,r}(\bm{Z}_{r,1}) \wconv
h_1(\bm{Z}).
\end{equation*} 
Moreover, for any $p < 2 + \nu$ with $p\in\N$, we have $\lim_{r\to\infty} \Exp[h^p_{1,r}(\bm{Z}_{r,1})]= \Exp[h^p_1(\bm{Z})]$.
\end{lemma}

\begin{proof} Since $\vartheta_r \to \vartheta$ by Lemma~\ref{lem:varthetar},  we may assume $\vartheta_r \equiv 0$. 
We will use Wichura's Theorem \citep[Theorem 4.2]{bill_conv_pm}. Note that
\[
T_r := h_{1,r}(\bm Z_{r,1})
=
\int h(\bm{Z}_{r,1},\bm{y}) \, \dPm_{\bm{Z}_{r,1}}(\bm{y}),
\quad 
T:= h_1(\bm{Z}) = \int h(\bm{Z}, \bm{y}) \, \dPm_{\bm{Z}}(\bm{y})
\]
and define, for $B := B(b) := [-b,b]^d$ with $b\in\N$,
\begin{align*}
T_r(b) := \int_B h(\bm{Z}_{r,1}, \bm{y}) \, \dPm_{\bm{Z}_{r,1}}(\bm{y}), \quad
T(b) := \int_B h(\bm{Z}, \bm{y}) \,\dPm_{\bm{Z}}(\bm{y}).
\end{align*}

In order to show weak convergence of $T_r(b)$ to $T(b)$ we use the extended continuous mapping theorem (Theorem 1.11.1 in \cite{VanWel96}). Let $\bm{x}_r \to \bm{x} \in \R^d$ and note that the map $(\bm{x},\bm{y}) \mapsto h(\bm{x},\bm{y})\ind{\bm{y} \in B}$ is $\Pm_{\bm{Z}}^{\otimes 2}$-a.e.\ continuous. By the ordinary continuous mapping theorem we obtain weak convergence of
$h(\bm{x}_r, \bm{Z}_{r,1})\ind{\bm{Z}_{r,1} \in B}$ to $h(\bm{x},\bm{Z})\ind{\bm{Z} \in B}.$ 
Next, since there exists a compact set $A$ containing $(\bm x_r)_r$,  we have
\[
\limsup_{r\to\infty} \Exp[h^2(\bm{x}_r,\bm{Z}_{r,1}) \cdot \ind{\bm{Z}_{r,1} \in B}] 
\leq \sup_{\bm{x} \in A, \bm{z} \in B}h^2(\bm{x}, \bm{z}) < \infty,
\]
which in turn implies moment convergence of $h(\bm{x}_r, \bm{Z}_{r,1})\ind{\bm{Z}_{r,1} \in B}$. This shows continuous convergence of the mapping sequence $\bm{x} \mapsto \int_B h(\bm{x}, \bm{y}) \on{d}\!\Pm_{\bm{Z}_{r,1}}(\bm y)$, and the extended continuous mapping theorem finally implies weak convergence of $T_r(b)$ to $T(b)$ as asserted.

Next, we have weak convergence of $T(b)$ to $T$, for $b\to\infty$. Indeed, with $\tilde{\bm{Z}}$ an independent copy of $\bm{Z}$, we have 
\begin{align*}
\Exp|T-T(b)| 
\leq  
\Exp\big[ \abs{h(\bm{Z}, \tilde{\bm{Z}})} \bm 1\{ \tilde{\bm{Z}} \in B(b)\} \big]
&\leq 
\norm{h(\bm{Z}, \tilde{\bm{Z}})}_{L^2(\mathbb P)} \cdot \Pm ( \tilde {\bm{Z}} \not \in B(b) )^{1/2} \\
&=o(1)
\end{align*}
as $b\to\infty$.

We finally verify $\lim_{b\to\infty} \limsup_{r\to\infty} \Pm\left( \abs{T_r(b) - T_r} > \varepsilon \right)=0$ for any fixed $\varepsilon > 0$. Let $\tilde{\bm{Z}}_{r,1}$ be an independent copy of $\bm{Z}_{r,1}.$ Applying the Markov inequality, we have
\begin{align*}
\Pm\left( \abs{T_r(b) - T_r} > \epsilon \right) 
&\leq \Pm \left( \int_{B^c} \abs{h(\bm{Z}_{r,1}, \bm{y})} \dPm_{\bm{Z}_{r,1}}(\bm{y}) > \varepsilon \right) \\
&\leq   \varepsilon^{-1} \Exp[|h(\bm{Z}_{r,1}, \tilde{\bm{Z}}_{r,1})| \bm 1\{\tilde{\bm{Z}}_{r,1} \in B(b)^c\}].
\end{align*}
Applying the Cauchy-Schwarz inequality and taking the limit over $r$ results in the upper bound $C \cdot \Pm\left( \bm{Z}\in B(b)^c \right)^{1/2}$ by the Portmanteau Theorem and Condition~\ref{cond:h_int_std}~(a), for a constant $C$ not depending on $b$. The bound goes to 0 for $b\to\infty$ since $B(b)^c \downarrow \emptyset.$

Wichura's theorem implies weak convergence of $h_{1,r}(\bm{Z}_{r,1})$ to  $h_1(\bm{Z})$, for $r \to\infty$. The stated convergence of moments follows by Example 2.21 in \cite{Van98} using the Jensen inequality and Condition~\ref{cond:h_int_std}~(a).
\end{proof}

Let $\ell = \ell_n \in \N$ denote the sequence from Condition~\ref{cond:ser_dep}(b). We may assume that $1<\ell<r$. 
For $j\in\N$, recall that
\begin{align} \label{eq:mlj}
    \bm{M}_{r-\ell,j} = \max (\bm{X}_{j}, \ldots \bm{X}_{j+r-\ell-1}) , \quad 
    \bm{Z}_{r-\ell,j} = (\bm{M}_{r-\ell,j}-\bm{b}_{r-\ell})/\bm{a}_{r-\ell}.
\end{align}

\begin{lemma}[Weak convergence of clipped blocks]
\label{lem:joint_conv_short_blocks}
Suppose Conditions \ref{cond:ser_dep}(a) and (b) are met. Then, as $n\to\infty$, 
\begin{equation*}
(\bm{Z}_{r,1}, \bm{Z}_{r-\ell,1}) \wconv (\bm{Z}, \bm{Z}).
\end{equation*}
\end{lemma}

\begin{proof}
Since
$(\bm{Z}_{r,1}, \bm{Z}_{r-\ell,1}) 
=  (\bm{Z}_{r,1}, \bm{Z}_{r,1}) -
(0,\bm{Z}_{r-\ell,1} - \bm{Z}_{r,1})$ and since $\bm Z_{r,1}$ converges weakly to $\bm Z$ by assumption, it suffices to show that $ \bm{Z}_{r-\ell,1} - \bm{Z}_{r,1} = o_\Pm(1)$. In particular, we may assume $d=1$ and note $(M_{r-\ell,1} - b_{r-\ell})/a_{r-\ell}$ converges weakly to $\bm Z$.

Condition \ref{cond:mda} yields local uniform convergence, see the proof of Lemma B.15 in \cite{BucZan23}, hence $a_r/a_{r-\ell} = 1+o(1)$ and $(b_{r-\ell}-b_r)/a_r = o(1)$. 
By Lemma B.15 from \cite{BucZan23} we have, for any $\varepsilon > 0$,
\[
\Pm\left(\abs{ Z_{r-\ell,1} - Z_{r,1} } \geq \varepsilon \right) = \Pm \left( \abs{ Z_{r-\ell,1} - Z_{r,1} } \geq \varepsilon, M_{r-\ell,1} = M_{r,1} \right) +o(1).
\]
Using the convergence of the rescaling sequences and that $Z_{r-\ell}$ is stochastically bounded we have
\begin{align*}
\frac{M_{r-\ell,1}-b_{r-\ell}}{a_{r-\ell}} - \frac{M_{r-\ell,1}-b_r}{a_r} 
&=  
Z_{r-\ell,1} \left( 1 - \frac{a_{r-\ell}}{a_r} \right) - \frac{b_{r-\ell} - b_r} {a_r} \\
&= 
O_\Pm(1) o(1) + o(1) = o_\Pm(1).
\end{align*}
This implies $ Z_{r-\ell,1} - Z_{r,1} = o_\Pm(1).$
\end{proof}

For the next results, let $\Delta_{r,\ell}(j) := h_{1,r}(\bm{Z}_{r,j}) - h_{1, r- \ell}(\bm{Z}_{r-\ell,j})$ for $j \in \Idb$.
Furthermore let $(\bm{\tilde X}_j, \dots, \bm{\tilde X}_{j+r-1})_{j\in \Idb}$ be i.i.d.\ copies of $(\bm{X}_j, \dots, \bm{X}_{j+r-1})_{j\in\Idb}$ 
and define  $\tilde{\bm M}_{r,j}=\max(\tilde{\bm X}_j, \dots, \tilde {\bm X}_{j+r-1})$ and $\tilde{\bm Z}_{r,j}=(\tilde{\bm M}_{r,j}-\bm b_r)/\bm a_r$ and $\tilde{\bm{M}}_{r-\ell,j}, \tilde{\bm{Z}}_{r-\ell,j}$ analogously to \eqref{eq:mlj}.  

\begin{lemma}\label{lem:block_diff_L^2_conv}
Suppose Conditions \ref{cond:ser_dep}(a), (b) and \ref{cond:h_int_std}(a) are met and that $h$ is $\dlambda^{2d}$-a.e.\ continuous. Then 
\[
\lim_{n\to\infty} \Exp\Big[ \Big\{ \frac{1}{\sqrt{m}} \sum_{j \in \Idb} \Delta_{r,\ell}(j) \Big\}^2 \Big]=0,
\qquad
\lim_{n\to\infty}
\Exp\Big[ \big| \Delta_{r,\ell}(1) \big|^p \Big] = 0
\]
for all $p \in \N$ with $p < 2+ \nu.$
\end{lemma}

\begin{proof}
We start by showing the second convergence. Let $\| \cdot \|_p$ denote the $L_p$-norm. Writing $\Delta_{r,\ell}(1) = \int h(\bm{Z}_{r,1},y_1) - h(\bm{Z}_{r-\ell,1}, y_2) \on{d}\!\Pm_{\bm{\tilde{Z}}_{r,1}} \otimes \Pm_{\bm{\tilde{Z}}_{r-\ell,1}} (y_1, y_2) + \vartheta_{r-\ell}-\vartheta_r$, we obtain
\begin{align*}
 \norm{\Delta_{r,\ell}(1)}_p 
\leq  \norm{h(\bm{Z}_{r,1}, \tilde{\bm{Z}}_{r,1} ) - 
h(\bm{Z}_{r-\ell,1}, \tilde{\bm{Z}}_{r-\ell,1} ) }_p
 + \abs{\vartheta_r - \vartheta_{r-\ell}} =: r_{1,n} + r_{2,n},
\end{align*}
by Jensen's inequality.
Since $\lim_{r \to\infty} \vartheta_r = \vartheta$, we have $r_{2,n} = o(1)$ for $n\to\infty$. 
By Lemma \ref{lem:joint_conv_short_blocks} and independence, the vector
$(\bm{Z}_{r,1}, \bm{Z}_{r-\ell,1}, \bm{\tilde{Z}}_{r,1}, \bm{\tilde{Z}}_{r-\ell,1})$ converges weakly to $ (\bm{Z} , \bm{Z},\bm{\tilde{Z}} ,\bm{\tilde{Z}})$, where $\bm Z, \tilde{\bm Z}$ are i.i.d.\ with cdf $G$. 
Therefore, the continuous mapping theorem yields
$ \abs{h(\bm{Z}_{r,1}, \tilde{\bm{Z}}_{r,1} ) - 
h(\bm{Z}_{r-\ell,1}, \tilde{\bm{Z}}_{r-\ell,1} )} =o_{\mathbb P}(1)$. By Condition \ref{cond:h_int_std}(a) we have asymptotic uniform integrability of $|h(\bm{Z}_{r,1}, \tilde{\bm{Z}}_{r,1} )|^{2+\nu}$ so that $r_{1,n} = o(1)$ by Example 2.21 in \cite{Van98} and stationarity.

Using stationarity and observing that the $\Delta_{r,\ell}(j)$ are centered we have 
\begin{multline*}
\Exp\Big[ \Big\{ \frac{1}{\sqrt{m}} \sum_{j \in \Idb} \Delta_{r,\ell}(j) \Big\}^2 \Big]  \\
\leq  3\Var(\Delta_{r,\ell}(1)) + 2 \sum_{s=2}^{m-1} (1-\frac{s}{m}) |\Cov(\Delta_{r,\ell}(1), \Delta_{r,\ell}(1+rs))|.
\end{multline*}
By Lemma 3.11 in \cite{DehPhi02}, Condition \ref{cond:h_int_std}(a) and \ref{cond:ser_dep} (c), there exists a constant $C > 0$ that is independent of $s \ge 2$ and $n$ such that $
\abs{\Cov(\Delta_{r,\ell}(1), \Delta_{r,\ell}(1+rs))} 
\leq C \alpha(r)^{v/(2+\nu)}.
$
Now $\alpha(r) \leq \beta(r)$ and Condition~\ref{cond:ser_dep} (c) imply that the sum in the upper display converges to 0. Hence, an application of the second claim of this lemma implies the first claim and the proof is finished.
\end{proof}

\begin{lemma}[Restriction to independent blocks] \label{lem:switch_to_indep}
Suppose Conditions \ref{cond:ser_dep}(a), (b), \ref{cond:h_int_std}(a) are met and $h$ is $\dlambda^{2d}$-a.e.\ continuous. Then $m^{-1/2}\sum_{j \in \Idb} h_{1,r}(\bm{Z}_{r,j})$ converges weakly if and only if $m^{-1/2} \sum_{j \in \Idb} h_{1,r}(\tilde{\bm{Z}}_{r,j})$ converges weakly. In that case the weak limits coincide.
\end{lemma}

\begin{proof}
Let $\phi_Y$ denote the characteristic function of a real-valued random variable $Y.$ First note that, by Lemma 3.9 in \cite{DehPhi02}, we have, for $t \in \R$,
\begin{align*}
&\phantom{{}={}} \abs{\phi_{h_{1,r-\ell}(\bm{Z}_{r-\ell,1}) + h_{1,r-\ell}(\bm{Z}_{r-\ell,1+r})}(t) - \phi_{h_{1,r-\ell}(\tilde{\bm{Z}}_{r-\ell,1}) + h_{1,r-\ell}(\tilde{\bm{Z}}_{r-\ell,1+r})}(t)} \\
&=  \big| \Cov \big( e^{ith_{1,r-\ell}(\bm{Z}_{r-\ell,1})},e^{ith_{1,r-\ell}(\bm{Z}_{r-\ell,1+r})} \big) \big| 
\leq 2\pi \alpha(\ell),
\end{align*}
as there is a lag of length $\ell$ between $\bm{X}_{r-\ell}$ and $\bm{X}_{r+1}.$
An induction over the number of summands results in 
\begin{equation*} 
\Big| \phi_{\sum_{j \in \Idb} h_{1,r-\ell}(\bm{Z}_{r-\ell,j}) }(t) - \phi_{\sum_{j \in \Idb} h_{1,r-\ell}(\tilde{\bm{Z}}_{r-\ell,j})}(t) \Big|  
\leq 2\pi m\alpha(\ell) =o(1),
\end{equation*}
by Condition \ref{cond:ser_dep} (b). Hence, by Levy's continuity theorem, the weak limit of $m^{-1/2}\sum_{j \in \Idb}h_{1,r-\ell}(\bm{Z}_{r-\ell,j})$ exists iff the weak limit of $m^{-1/2}\sum_{j \in \Idb} h_{1,r-\ell}(\tilde{\bm{Z}}_{r-\ell,j})$ exists, and in that case, the limits  coincide.

Finally, by Lemma \ref{lem:block_diff_L^2_conv}
\begin{align*}
\frac{1}{\sqrt{m}} \sum_{j \in \Idb} h_{1,r}(\bm{Z}_{r,j}) 
&=  
\frac{1}{\sqrt{m}} \sum_{j \in \Idb}  \Delta_{r,\ell}(j) +  \frac{1}{\sqrt{m}} \sum_{j \in \Idb} h_{1,r-\ell}(\bm{Z}_{r-\ell,j}) \\
&= 
o_\Pm(1) +   \frac{1}{\sqrt{m}} \sum_{j \in \Idb} h_{1,r-\ell}(\bm{Z}_{r-\ell,j}).
\end{align*}
Thus, the weak limit of $m^{-1/2}  \sum_{j \in \Idb} h_{1,r}(\bm{Z}_{r,j})$ exists iff the weak limit of $m^{-1/2} \sum_{j \in \Idb} h_{1,r-\ell}(\bm{\tilde{Z}}_{r-\ell,j})$ exists and they coincide. The same line of argumentation as in the proof Lemma \ref{lem:block_diff_L^2_conv} yields the result.
\end{proof}

\subsection{Sliding blocks - stationary case}

Recall the definitions of $G_\xi$, $L_\xi$ and  $C_{\xi}$ from \eqref{eq:gxi}, \eqref{eq:lxi} and \eqref{eq:cxi}, respectively. Recall the convention that $L := \on{id}_{[0,\infty]}$ if $d=1$, which implies $C=\on{id}_{[0,1]}$.
The following is a generalization of Lemma B.3 in the supplementary material to \cite{BucZan23} for dimensions $d \geq 1$.

\begin{lemma}\label{lem:overlap_wconv}
Suppose that Conditions \ref{cond:ser_dep}(a) and (b) are met. Then, for any $\xi \geq 0$ and $\bm{x},\bm{y} \in \R^d$,
\begin{equation*}
\lim_{n \to \infty} \Pm(\bm{Z}_{r,1} \leq \bm{x}, \bm{Z}_{r,\xi_r} \leq \bm{y} ) = G_\xi(\bm{x}, \bm{y}),
\end{equation*}
where $\xi_r = 1+ \flo{r\xi}$.
Furthermore, $G_\xi$ is the cdf of a $2d$-variate extreme value distribution with copula $C_{\xi}$ and stable tail dependence function $L_\xi$.
\end{lemma}
\begin{proof}
We only consider the case $\xi \in [0,1]$; the case $\xi > 1$ can be treated similarly. By the same arguments as in the proof of Lemma~B.3 in \cite{BucZan23}, we obtain that
\begin{multline} \label{eq:gxi2}
\lim_{n \to \infty} \Pm(\bm{Z}_{r,1} \leq \bm{x}, \bm{Z}_{r,\xi_r} \leq \bm{y} )     \\
=  
    G\left( \xi^{-\gamma^{(1)}}x^{(1)} +\frac{\xi^{-\gamma^{(1)}}-1}{\gamma^{(1)}}, \ldots, \xi^{-\gamma^{(d)}}x^{(d)} +\frac{\xi^{-\gamma^{(d)}}-1}{\gamma^{(d)}}\right)  \\
    \times G\left( \xi^{-\gamma^{(1)}}y^{(1)} + \frac{\xi^{-\gamma^{(1)}}-1}{\gamma^{(1)}}, \ldots, \xi^{-\gamma^{(d)}}y^{(d)} + \frac{\xi^{-\gamma^{(d)}}-1}{\gamma^{(d)}} \right) \\ 
    \times  G\Bigg( (1-\xi)^{-\gamma^{(1)}} \left( x^{(1)} \wedge y^{(1)} \right) + \frac{(1-\xi)^{-\gamma^{(1)}}-1}{\gamma^{(1)}}, \ldots \\
    \ldots,(1-\xi)^{-\gamma^{(d)}} \left(x^{(d)} \wedge y^{(d)} \right) + \frac{(1-\xi)^{-\gamma^{(d)}}-1}{\gamma^{(d)}} \Bigg).
\end{multline}
Since $-\log G_\gamma(x) = (1+\gamma x)^{-1/\gamma}$, we may write
\begin{align*}
G(\bm x) 
&=
\exp\Big\{-L\Big( - (1+\gamma^{(1)} x^{(1)})^{-1/\gamma^{(1)}}, 
\dots,
(1+\gamma^{(d)} x^{(d)})^{-1/\gamma^{(d)}} \Big) \Big\},
\end{align*}
A straightforward calculation then shows that the expression on the right-hand side of \eqref{eq:gxi2} can be written as $G_\xi(\bm x, \bm y)$. In particular, $C_\xi$ is a copula, which can easily be seen to be max-stable, i.e., $C_\xi(\bm u^s, \bm v^s) = C_\xi(\bm u, \bm v)^s$ for all $s>0$ and $\bm u, \bm v \in [0,1]^d$. It is hence an extreme-value copula with the given stable tail-dependence function $L_\xi$ and $G_\xi$ is the cdf of an extreme-value distribution.
\end{proof}

\begin{theorem} [CLT for sliding blocks] \label{thm:slid_gws}
Suppose that Conditions~\ref{cond:ser_dep}(a), (b) are satisfied and that there exists an $\omega > 0$ with $m_n^{1+\omega}\alpha(r_n) \to 0.$ For each $r=r_n$ let $\mc{F}_r = \lbrace f_{r,i} \colon \R^d \to \R \mid i \in \N \rbrace$ be a family of deterministic maps with the following properties:
 \begin{enumerate}[(i)]
 \item $f_{r,r+s} = f_{r,s}$ for all $s \in \N$ and $r=r_n$ with $n\in\N$;
 \item The random variables $f_{r,i}(\bm{Z}_{r,i})$ are centered for all $i \in \N$ and $r=r_n$ with $n\in\N$;
 \item There exists a $\nu > 2/\omega$ with $\limsup_{n \to \infty} \sup_{i\in\N} \Exp[ |f_{r,i}(\bm{Z}_{r,i})|^{2+\nu}] < \infty.$
 \item There exists a map $f \colon \R^d \to \R$ such that for, all $\xi \in [0,1], \xi^\prime \in [0,2]$, as $n \to \infty,$
\begin{align} \label{eq:wl2}
\left(f_{r, \xi_r}(\bm{Z}_{r,\xi_r}),f_{r, \xi'_r}(\bm{Z}_{r,\xi'_r}) \right) \wconv 
\left( f(\bm{Z}_{1, \abs{\xi - \xi^\prime}}), f(\bm{Z}_{2, \abs{\xi - \xi^\prime}}) \right),
\end{align}
where $\xi_r=1+\flo{r\xi}, \xi'_r=1+\flo{r\xi'}$ and $(\bm{Z}_{1, \abs{\xi - \xi^\prime}}, \bm{Z}_{2, \abs{\xi - \xi^\prime}}) \sim G_{\abs{\xi - \xi^\prime}}$.
 \end{enumerate}
Then, for $n\to\infty$,
\begin{align*}
\frac{\sqrt{m}}{n} \sum_{i=1}^{n} f_{r,i}( \bm{Z}_{r,i})\wconv 
\mathcal N(0, \sigma_f^2), \quad
\sigma_f^2 :=
2\int_0^1\Cov(f(\bm{Z}_{1, \xi}), f(\bm{Z}_{2, \xi})) \diff\xi.   
\end{align*}
\end{theorem}

\begin{proof}
The proof is similar to the one of Theorem 3.6 in \cite{BucSeg18a}.
For $j\in\{1,\ldots,m\}$, let  $I_j := \lbrace (j-1)r+1,\ldots,jr \rbrace$. Choose $m^\ast = m_n^\ast \in \N$ with $3 \leq m^\ast \leq m$ such that $m^\ast \to \infty$ and $m^\ast = o(m^{\nu/(2(1+\nu))}).$ Next, define $q := q_n := m/m^\ast$ and assume without loss of generality that $q \in \N, n/r \in \N$. For $j\in\N$ define $J_j^+ := I_{(j-1)m^\ast + 1} \cup \ldots \cup I_{jm^\ast-2}$ as the index set making up the big blocks, and $J_j^- := I_{jm^\ast -1} \cup I_{jm^\ast}$ as the index set making up the small blocks. Note that $\#J_j^+ = (m^\ast -2)r$ and $\#J_j^- = 2r.$ The previous definitions allow to rewrite
\begin{equation}\label{eq:proof_clt_sld_decomp}
\frac{\sqrt{m}}{n} \sum_{i=1}^{n} f_{r,i}(\bm{Z}_{r,i}) = \frac{1}{\sqrt{q}} \sum_{j=1}^q ( S_{n,j}^+ + S_{n,j}^- ),
\end{equation}
where $S_{n,j}^+ := \sqrt{q/(nr)} \sum_{s \in J_j^+} f_{r,s}(\bm{Z}_{r,s})$ and $S_{n,j}^- := \sqrt{q/(nr)} \sum_{s \in J_j^-} f_{r,s}(\bm{Z}_{r,s}).$

Note that the random variables $(S_{n,j}^{\pm})_j$ are stationary by (i). Hence, by (ii)
\begin{align*}
\Var\Big( \frac1{\sqrt q}\sum_{j=1}^q S_{n,j}^- \Big)
&= 
\frac{1}{q} \sum_{j=1}^q \Var(S_{n,j}^-) + \frac{2}{q} \sum_{1 \leq i < j \leq q} \Cov( S_{n,i}^-, S_{n,j}^- ) \\ 
&\leq
3 \Var ( S_{n,1}^-)  + 2 \sum_{k=2}^{q-1} |\Cov ( S_{n,1}^-, S_{n,1+k}^-)| =: R_{n1} +R_{n2}.
\end{align*}
Properties (ii), (iii) and the definition of $m^*$ and $q$ yield 
\begin{align*}
\left(\frac{R_{n1}}{3}\right)^{1/2} \leq \sqrt{\frac{q}{nr}} \sum_{s \in J_1^-} \norm{f_{r,s}(\bm{Z}_{r,s})}_2
= O\Big(\sqrt{\frac{1}{m^\ast}}\Big) =o(1).
\end{align*}
For $R_{n,2}$ note that by property (iii) and Lemma 3.11 from \cite{DehPhi02} we have that $\sup_{k\ge 2} |\Cov(S_{n,1}^-, S_{n,1+k}^-) | \lesssim (m^\ast)^{-1} \alpha(r)^{\nu/(2+\nu)}$. Since $m^\ast \geq 3$ we obtain $R_{n,2} \lesssim m (m^\ast)^{-2}\alpha(r)^{\nu/(2+\nu)} = o(1)$  by assumption. Therefore, in view of $\Exp[S_{n,j}^-]= 0$ for all $j$ by (ii), we obtain $q^{-1/2}\sum_{j=1}^q S_{n,j}^- = o_\Pm(1)$.

Concerning the sum over $S_{n,j}^+$ note that we may assume that $S_{n,1}^+, S_{n,2}^+, \dots $ are independent by arguing as in the proof of Lemma \ref{lem:switch_to_indep}, since there is a lag of $r$ between any two big blocks. Hence, we may subsequently apply Ljapunov's central limit theorem. 

We will show below that 
\begin{align}
\label{eq:csigma} 
\lim_{n\to\infty} \sigma_n^2 = \sigma_f^2, \qquad
\sigma_n^2 :=\Var\Big(\frac{\sqrt{m}}{n} \sum_{i=1}^{n} f_{r,i}\left(\bm{Z}_{r,i}\right)\Big).
\end{align}
If $\sigma_f^2=0$, we immediately obtain the assertion. If $\sigma_f^2>0$, we obtain 
\begin{align*}
\Var \Big(\sum_{j=1}^q S_{n,j}^+\Big)
&=  
q \Big \lbrace \Var\Big( \frac{\sqrt{m}}{n} \sum_{i=1}^{n} f_{r,i}(\bm{Z}_{r,i}) \Big) +o(1) \Big \rbrace 
=  
q\{ \sigma_f^2 + o(1) \},
\end{align*}
where the $o$ term is due to \eqref{eq:proof_clt_sld_decomp} and $\Var(m^{-1/2} \sum_{j=1}^q S_{n,j}^-) = o(1).$
Moreover, by property (iii), we have $\sup_{j \in \N} \Exp[\abs{S_{n,j}^+}^{2+\nu}] = O((m^\ast)^{1+\nu/2})$. As a consequence,
\begin{align*}
\frac{\sum_{j=1}^q \Exp[ \abs{S_{n,j}^+}^{2+\nu} ]}{ \{ \Var(\sum_{j=1}^q S_{n,j}^+) \}^{1+\nu/2} 
}
\lesssim 
\frac{q(m^\ast)^{1+\nu/2}}{(q(\sigma_f^2+o(1)))^{1+\nu/2}} 
\lesssim \frac{(m^\ast)^{1+\nu}}{m^{v/2}} = o(1)
\end{align*}
by choice of $m^\ast.$ Ljapunvov's central limit theorem implies the assertion.

It remains to prove \eqref{eq:csigma}. For $k\in \{1, \dots, m \}$, let $A_k := \sum_{s \in I_k} f_{r,s}(\bm{Z}_{r,s})$ and note that $\sum_{i=1}^{n}f_{r,i}(\bm{Z}_{r, i}) = \sum_{k=1}^{m}A_k$. Here, Assumption (i) implies stationarity of $(A_k)_k$, whence
\[
\sigma_n^2
= 
\frac{1}{nr} \Var\Big(\sum_{k=1}^{m}A_k\Big)
=
\frac{m}{nr} \{ \Var(A_1) + 2\, \Cov(A_1, A_2)\} + \frac{R_n}{nr}, 
\]
where $R_n := -2\Cov(A_1, A_2) + 2\sum_{k=2}^{m-1}(m-k)\Cov(A_1, A_{1+k})$. Lemma 3.11 in \cite{DehPhi02}, together with Assumptions (ii) and (iii), implies that 
\begin{align*}
\frac{1}{nr} |R_n| 
&\lesssim 
\frac{r^2}{nr} + \frac{m^2 r^2}{nr} \alpha^{\frac{\nu}{2+\nu}}(r) 
\lesssim
\frac{1}{m} + \left( m^{1+2/\nu} \alpha(r) \right)^{\nu/(2+\nu)} = o(1),
\end{align*}
by assumption. Hence $\sigma^2_n = r^{-2} \{\Var(A_1) + 2\Cov(A_1, A_2)\} + o(1)$. Define, for $\xi, \xi^\prime \geq 0$,
\[
g_n(\xi, \xi^\prime) := \Cov\big(f_{r, \xi_r}(\bm{Z}_{r, \xi_r}) , f_{r, \xi'_r}(\bm{Z}_{r, \xi'_r})\big),
\]
and note that, by \eqref{eq:wl2}, Assumption (iii) and the continuous mapping theorem, 
\[
\lim_{n \to \infty} g_n(\xi, \xi^\prime) 
=
g(\xi, \xi^\prime) 
:= 
\Cov \big( f(\bm{Z}_{1,|\xi -\xi^\prime|} ), f(\bm{Z}_{2,|\xi -\xi^\prime|} ) \big).
\]
The dominated convergence theorem implies
\[
\frac{\Var(A_1)}{r^2} = \int_0^1\int_0^1 g_n(\xi, \xi^\prime) \on{d}\!\xi \on{d}\!\xi^\prime \to 
\int_0^1\int_0^1 g(\xi, \xi^\prime) \on{d}\!\xi \on{d}\!\xi^\prime,
\]
as by (ii) and (iii) $g_n(\xi, \xi^\prime)$ may be bounded uniformly in $n, \xi, \xi^\prime.$ 
Similarly we obtain $r^{-2} \Cov(A_1, A_2) \to \int_0^1 \int_1^2 g(\xi, \xi^\prime) \on{d}\!\xi \on{d}\!\xi^\prime.$ We may finally proceed as in the proof of Lemma B.9
in the supplement of \cite{BucZan23} to obtain $\lim_n \sigma^2_n = \sigma^2_f.$
\end{proof}

\subsection{Sliding blocks - non-stationary case}

The following theorem is an adaptation of Theorem~\ref{thm:slid_gws} to the non-stationary setting of Section~\ref{sec:ext}.

\begin{theorem}\label{thm:slid_gws-ns}
Suppose that the sampling scheme from Condition~\ref{cond:s2} is met and that the underlying time-series $(\bm Y_t)_t$ satisfies Conditions~\ref{cond:ser_dep}(a), (b) and  $m_n^{1+\omega}\alpha(r_n) \to 0$ for some $\omega>0$.
For each $r=r_n$, let $\mc{F}_r = \lbrace f_{r,i} \colon \R^d \to \R \mid i \in \N \rbrace$ be a family of deterministic maps satisfying Conditions (i) - (iv) of Theorem~\ref{thm:slid_gws}. Then,
for $n \to \infty,$
\begin{align*}
\frac{\sqrt{m}}{n} \sum_{i=1}^{n} f_{r,i}( \bm{Z}_{r,i})\wconv 
\mathcal N(0, \sigma_f^2), \quad
\sigma_f^2 :=
2\int_0^1\Cov(f(\bm{Z}_{1, \xi}), f(\bm{Z}_{2, \xi})) \diff\xi.   
\end{align*}
\end{theorem}

\begin{proof}
The proof is essentially the same as for Theorem~\ref{thm:slid_gws}, with the following simple adaptation: independence of $S_{n,1}^+, S_{n,2}^+, ...$ is a direct consequence of the imposed sampling scheme.
\end{proof}

The following result is an extension of Lemma 2.4 from \cite{BucZan23} to multivariate time series. 

\begin{lemma}\label{lem:as_stat_Zrxi}
Suppose the sampling scheme from Condition \ref{cond:s2} is met and that the underlying time series $(\bm Y_t)_t$ satisfies Conditions \ref{cond:ser_dep}(a) and (b). Then, for every $\xi \geq 0$ and $\bm x \in \R^d$,
\[
\lim_{n \to \infty} \Pm( \bm{Z}_{r, \xi_r} \leq \bm x) = G(\bm x),
\]
with $G$ from Condition~\ref{cond:mda} and with $\xi_r := 1+ \flo{r\xi}$.
\end{lemma}

\begin{proof}
Note first that for univariate $x \in \R$ and $s > 0, \gamma \in \R$ we have $G_\gamma(x/s^\gamma + (s^{-\gamma}-1)/\gamma) = G_\gamma(x)^s.$ 
This implies, for $\bm x \in \R^d, \bm \gamma \in \R^d$, 
\begin{align*}
 G\Big[ \Big(\tfrac{x^{(i)}}{s^{\gamma^{(i)}}} + \tfrac{s^{-\gamma^{(i)}} -1}{\gamma^{(i)}}\Big)_{i=1,\ldots,d}\Big] 
  &=  C \Big[ \Big(G_{\gamma^{(i)}}\Big(\tfrac{x^{(i)}}{s^{\gamma^{(i)}}} + \tfrac{s^{-\gamma^{(i)}} -1}{\gamma^{(i)}}\Big)\Big)_{i=1,\ldots,d}\Big]\\
& =  C^s \Big[\Big(G_{\gamma^{(i)}}\big(x^{(i)}\big)\Big)_{i=1,\ldots d}\Big]
=  G^s(\bm x),
\end{align*}
by \eqref{eq:evc} and (L1) from Condition~\ref{cond:mda}.

By piecewise stationarity and Condition \ref{cond:mda} it suffices to show the result for $\xi \in (0,1).$ Analogous to the proof of Lemma 2.4 from \cite{BucZan23} we have 
\begin{align*}
&\phantom{{}={}} \lim_{r \to \infty} \Pm( \bm{Z}_{r, \xi_r} \leq \bm x) \\
&=  G\Big[ \Big(\tfrac{x^{(i)}}{(1-\xi)^{\gamma^{(i)}}} + \tfrac{(1-\xi)^{-\gamma^{(i)}} -1}{\gamma^{(i)}}\Big)_{i = 1, \ldots, d}\Big]
\cdot G\Big[ \Big(\tfrac{x^{(i)}}{\xi^{\gamma^{(i)}}} + \tfrac{\xi^{-\gamma^{(i)}} -1}{\gamma^{(i)}}\Big)_{i=1, \ldots, d}\Big]\\
&=  G(\bm x),
\end{align*}
where the last equality follows from the identity in the previous display.
\end{proof}

\begin{lemma}\label{lem:overlap_wconv_S2}
Suppose the sampling scheme from Condition \ref{cond:s2} is met and that the underlying time series $(\bm Y_t)_t$ satisfies Conditions \ref{cond:ser_dep}(a) and (b). Then, for any $\xi, \xi^\prime \geq 0$ and $\bm{x},\bm{y} \in \R^d$,
\begin{equation*}
\lim_{n \to \infty} \Pm(\bm{Z}_{r,\xi_r} \leq \bm{x}, \bm{Z}_{r,\xi^\prime_r} \leq \bm{y} ) = G_{|\xi-\xi^\prime|}(\bm{x}, \bm{y}).
\end{equation*}
\end{lemma}

\begin{proof}
This is a slight adaption of the proof of Lemma \ref{lem:overlap_wconv} using standard clipping techniques and Lemma \ref{lem:as_stat_Zrxi}. 
\end{proof}

\begin{lemma}\label{lem:overlap_h1_wconv_S2}
Suppose the sampling scheme from Condition \ref{cond:s2} is met and that the underlying time series $(\bm Y_t)_t$ satisfies Conditions \ref{cond:ser_dep}(a) and (b).  Furthermore, let $h$ be $\dlambda^{2d}$-a.e.\ continuous and bounded on compact sets and suppose that Condition \ref{cond:h_int_S2}(a) is met. Then, for $\xi, \xi^\prime \in [0, \infty)$
\begin{equation*}
\left( h_{1,r, \xi_r}(\bm{Z}_{r, \xi_r}), h_{1,r,\xi^\prime_r}(\bm{Z}_{r, \xi^\prime_r}) \right)
\wconv \left( h_1(\bm{Z}_{1, | \xi - \xi^\prime |}), h_1(\bm{Z}_{2, | \xi - \xi^\prime |}) \right)
\end{equation*}
and marginal moments of order $p < 2 +\nu,$ with $p \in \N,$ converge. 
\end{lemma}

\begin{proof}
We proceed similar as in the proof of Lemma \ref{lem:h_1,r_conv} and employ the Cramér-Wold Theorem and Wichura's Theorem. Fix $\bm a = (a^{(1)}, a^{(2)}) \in \R^2 \setminus \{\bm 0\}$ and let
    \begin{align*}
        T_n &:= a^{(1)} h_{1,r, \xi_r}(\bm Z_{r,\xi_r}) + a^{(2)} h_{1,r, \xi^\prime_r}(\bm Z_{r,\xi^\prime_r}) \\
        & =
        \int a^{(1)}h(\bm Z_{r,\xi_r},\bm y_1)+ a^{(2)}h(\bm Z_{r,\xi^\prime_r}, \bm y_2) \, \diff \Pm_{\bm{Z}_{r,\xi_r}} \otimes \Pm_{\bm{Z}_{r,\xi^\prime_r}}(\bm y_1, \bm y_2), \\
        T
        & := a^{(1)} h_{1}(\bm Z_{1, |\xi-\xi^\prime|}) + a^{(2)} h_{1}(\bm Z_{2, |\xi-\xi^\prime|}) \\
        & = \int a^{(1)} h(\bm Z_{1, |\xi-\xi^\prime|}, \bm y_1) + 
            a^{(2)} h(\bm Z_{2, |\xi-\xi^\prime|}, \bm y_2) \,\diff \Pm_{\bm{Z}}^{\otimes 2} (\bm y_1, \bm y_2)
    \end{align*}
    and define, for $B := B(b) := [-b,b]^d$ with $b\in\N$,
    \begin{align*}
        T_n(b) &:=  \int_{B \times B} a^{(1)}h(\bm Z_{r,\xi_r},\bm y_1)+ a^{(2)}h(\bm Z_{r,\xi^\prime_r}, \bm y_2) \, \diff \Pm_{\bm{Z}_{r,\xi_r}} \otimes \Pm_{\bm{Z}_{r,\xi^\prime_r}}(\bm y_1, \bm y_2), \\
        T(b) & := \int_{B \times B} a^{(1)} h(\bm Z_{1, |\xi-\xi^\prime|}, \bm y_1) + 
            a^{(2)} h(\bm Z_{2, |\xi-\xi^\prime|}, \bm y_2) \,\diff \Pm_{\bm{Z}}^{\otimes 2} (\bm y_1, \bm y_2).
    \end{align*}
    We may now proceed analogous to the proof of Lemma \ref{lem:h_1,r_conv}, where we use the extended continuous mapping theorem and the weak convergence from Lemma~\ref{lem:overlap_wconv_S2} to show that $T_n(b) \wconv T(b)$ for $n\to\infty$.
\end{proof}

\begin{acks}[Acknowledgments]
This work has been supported by the Deutsche Forschungsgemeinschaft (DFG, German Research Foundation), DFG project 465665892, which is gratefully acknowledged. 
\end{acks}

\bibliographystyle{imsart-number}
\bibliography{biblio}

\newpage

\begin{center}

\setlength{\baselineskip}{25pt}

{\huge \textbf{Supplement to \\
``Limit theorems for non-degenerate
U-statistics of block maxima for time
series''}}

\medskip

{\large \textbf{Axel Bücher and Torben Staud}}

\medskip

\setlength{\baselineskip}{10pt}

\end{center}

\noindent
\textbf{Abstract:} In Section~\ref{sec:ext_alpha} we provide an extension of Theorem~\ref{thm:U_conv} to strong mixing. In Section~\ref{sec:varformula}, we provide explicit formulas for some asymptotic variances from Section~\ref{sec:examples} in the main paper.

\appendix

\section{Limit results under strong mixing}\label{sec:ext_alpha}

The proof of Theorem~\ref{thm:U_conv} strongly relies on Berbee's coupling Lemma, which is a coupling result for beta-mixing time series \citep{berbee}. In the case of strong mixing, there is generally no similar coupling result that yields equality between the original and the coupling variables with high probability \citep{Deh83}. To the best of our knowledge, the strongest comparable result for alpha-mixing is due to \cite{Bra83}, which yields a coupling with a small $L^1$-distance. When deriving respective asymptotic results for U-statistics, Bradley's coupling construction makes it necessary to impose additional continuity assumptions on the kernel. Subsequently, we closely follow the concept of $\Pm$-Lipschitz continuity from \cite{dehlingPLip} which has been applied to U-statistics in the strongly mixing case in \cite{DehWen10a}. 

As common kernels with multivariate arguments do not satisfy the following regularity conditions, we will for reasons of simplicity assume $d=1$. 

\begin{condition}[Regularity of the kernel function] 
\label{cond:h_reg_alpha}
There exists a non-negative function $g \colon \R^3 \to [0,\infty)$ that is $\dlambda^3$-almost everywhere continuous and a $\kappa > 1$ such that the following two conditions are met:
\begin{compactenum}
\item[(a)] For all $(x_1,x_2,y) \in \R^3$, we have
\[  
\abs{h(x_1,y)-h(x_2,y)} \leq \abs{x_1 - x_2} g(x_1, x_2, y).
\]
\item[(b)] With $\mc{F}_r := \lbrace (Y_1,Y_2,Y_3) \colon \forall j\in\{1,2,3\} \, \exists i_j \in \N^3: \Pm_{Y_j} = \Pm_{Z_{r,i_j}}\rbrace$, where the $Y_i$ are random variables, we have
\[
\limsup_{r\to\infty} \sup_{(Y_1,Y_2,Y_3) \in \mc{F}_r} \Exp[g^{\kappa}(Y_1, Y_2, Y_3)] < \infty.
\]
\end{compactenum}
\end{condition}

The uniform integrability condition from Condition~\ref{cond:h_int_std} must be strengthened as follows.

\begin{condition}[Asymptotic integrability]\label{cond:h_int_alpha}
There exist $\nu, \rho > 0 $ with $\nu > 1/(\kappa -1) - 1$ and with  $\kappa$ from Condition \ref{cond:h_reg_alpha} such that Conditions \ref{cond:h_int_std}(a), (b) hold and
\begin{align}
\label{eq:hregal}
\limsup_{r \to \infty} \Exp[|Z_{r,1}|^{\rho}] < \infty.
\end{align}
\end{condition}

In specific applications, Conditions~\ref{cond:h_int_std}(a), (b) may often be reduced from moment constraints as in \eqref{eq:hregal}. Hence, Condition~\ref{cond:h_int_alpha} is not substantially stronger than  Condition~\ref{cond:h_int_std}.
Next, as we weaken the mixing requirements from absolute regularity to strong mixing we need the following stronger assumptions on the mixing rates.

\begin{condition}[Block size and serial dependence]\label{cond:ser_dep_alpha}
The block size sequence $(r_n)_n$ satisfies Conditions \ref{cond:ser_dep}(a) and (b). Moreover, $(n/r_n)^{3/2 + 2/\nu + 1/(2\rho)+ 1/(\rho \nu)} \alpha(r_n) = o(1)$, where $\rho$ and $\nu$ are from Condition \ref{cond:h_int_alpha}. 
\end{condition}

Finally, recalling the definitions of $U_{n,r}^{\mbl}$ in \eqref{eq:u_stat}, of $\theta_r$ in \eqref{eq:thetar},  of $\sigma^2_{\mbl}$ in \eqref{eq:as_vars} and of $\tilde \vartheta_r$ in \eqref{eq:varthetat}, we have the following result.

\begin{theorem}\label{thm:U_conv_alpha}
Suppose Conditions \ref{cond:ker_traf}, \ref{cond:h_reg_alpha}, \ref{cond:h_int_alpha} and \ref{cond:ser_dep_alpha} are met. Furthermore, let $h$ be $\dlambda^{2d}$-a.e.\ continuous and bounded on compact sets. Then, for $\mbl \in \{ \dbl, \sbl\}$, 
\begin{equation*}
\frac{\sqrt{m}}{ f(a_r, b_r)} \cdot ( \Unmb  - \theta_r) \wconv 
\Norm{0, \sigma_{\on{mb}}^2}.
\end{equation*}
If, additionally, the limit $B=\lim_{n\to \infty}B_n$ with $B_n$ from \eqref{eq:bn} exists, then
\begin{equation*}
\frac{\sqrt{m}}{f(a_r, b_r)} \cdot ( \Unmb  - \tilde\vartheta_r )\wconv 
\Norm{B, \sigma_{\on{mb}}^2}.
\end{equation*}
\end{theorem}

\begin{proof}[Proof of Theorem \ref{thm:U_conv_alpha}]
We use the same Hoeffding decompositon as in the proof of Theorem \ref{thm:U_conv}; see the algebraic identity (\ref{eq:proof_U_Zerl}). Since the proof of the asymptotic normality of $L_n^{\mbl}$ does not make use Condition \ref{cond:ser_dep}(c), the proof also applies in the current setting for $\alpha$-mixing. Moreover, for the disjoint blocks case, it is sufficient to show \eqref{eq:proof_u_conv_deg_part}.
By the same arguments as in the mentioned proof we may restrict attention to the sum over the indices from $\JJdb=\{(\bm i, \bm j) \in \Jdb \times \Jdb: \min(i_2 - i_1, j_1 - i_1) > 2r\}$.
By Condition \ref{cond:h_int_alpha} and Theorem 3 in \cite{Bra83}, after enlarging the probability space if necessary, there exist, for any $(\bm{i}, \bm{j}) \in \JJdb$, random variables $Z^\ast_{r,i_1}$ with the following properties:
\begin{compactenum}[(i)]
\item $Z^\ast_{r,i_1}$ is independent of $(Z_{r,i_2}, Z_{r,j_1}, Z_{r,j_2}),$
\item $Z_{r,i_1}$ and $Z^\ast_{r,i_1}$ have the same distributions,
\item $\forall \varepsilon>0$: $\Pm ( | Z_{r,i_1} - Z^\ast_{r,i_1} |  \geq \varepsilon ) \le K \alpha^{2\rho/(2\rho + 1)}(r) / \varepsilon^{\rho/(2\rho +1)}$
\end{compactenum}
where the constant $K$ does not depend on $(\bm{i}, \bm{j}) \in \JJdb$.
By the same arguments as in the proof of Theorem \ref{thm:U_conv} we have, for any $\varepsilon>0$, 
\begin{align}\label{eq:proof_u_conv_alpha_deg}
&\phantom{{}={}} 
\Exp\big[h_{2,r}(Z_{r,i_1}, Z_{r, i_2}) h_{2,r}(Z_{r,j_1}, Z_{r, j_2})\big]  \nonumber \\
&= \nonumber 
\Exp\big[\big\{ h_{2,r}(Z_{r,i_1}, Z_{r, i_2})- h_{2,r}(Z^\ast_{r,i_1}, Z_{r, i_2}) \big\} h_{2,r}(Z_{r,j_1}, Z_{r, j_2})\big] \\
&=\nonumber  
\Exp\big[\big\{ h_{2,r}(Z_{r,i_1}, Z_{r, i_2})- h_{2,r}(Z^\ast_{r,i_1}, Z_{r, i_2}) \big\} \\
&\hspace{4cm}\nonumber   \times h_{2,r}(Z_{r,j_1}, Z_{r, j_2}) \ind{| Z_{r,i_1} - Z^\ast_{r,i_1} |  < \varepsilon}\big] \\
& \nonumber  \hspace{1cm} + 
\Exp\big[\big\{ h_{2,r}(Z_{r,i_1}, Z_{r, i_2})- h_{2,r}(Z^\ast_{r,i_1}, Z_{r, i_2}) \big\} \\
&\hspace{4cm}\nonumber   \times h_{2,r}(Z_{r,j_1}, Z_{r, j_2}) \ind{| Z_{r,i_1} - Z^\ast_{r,i_1} |  \geq \varepsilon}\big] \\
&\equiv  R_{1,n} + R_{2,n}.
\end{align}
Note that 
\begin{multline*}
 h_{2,r}(Z_{r,i_1}, Z_{r, i_2})- h_{2,r}(Z^\ast_{r,i_1}, Z_{r, i_2}) \\ 
=  h(Z_{r,i_1}, Z_{r,i_2}) - h(Z^\ast_{r,i_1}, Z_{r,i_2}) + \int h(Z_{r,i_1}, y) - h(Z^\ast_{r,i_1}, y) \on{d}\Pm_{Z_{r,1}}(y),    
\end{multline*}
and apply Hölder's inequality with $p = 1+ 1/(1+\nu) < \kappa, \, q = 2+\nu$ to obtain
\begin{align} \label{eq:proof_u_conv_alpha_deg_R1n}
\phantom{{}={}} \nonumber  
|R_{1,n}| 
&\leq \nonumber 
\varepsilon \Exp\Big[ \Big\{  g(Z_{r,i_1}, Z^\ast_{r,i_1}, Z_{r, i_2})  +  \int g(Z_{r,i_1}, Z^\ast_{r,i_1}, y) \on{d}\Pm_{Z_{r,1}}(y) \Big\} \\
&\hspace{5cm}   \times \Big|h_{2,r}(Z_{r, j_1}, Z_{r, j_2}) \Big|\Big] \lesssim \varepsilon,
\end{align}
by Condition \ref{cond:h_reg_alpha}(a),(b) and Condition~\ref{cond:h_int_alpha}. 
Next, another application of Hölder's inequality with $p = (2+\nu)/2, \, q = (2+ \nu)/\nu$ yields 
\begin{align} \label{eq:proof_u_conv_alpha_deg_R2n}
& |R_{2,n}| \lesssim \Pm \left( | Z_{r,i_1} - Z^\ast_{r,i_1} |  \geq \varepsilon \right)^{\nu/(2+\nu)} 
\lesssim \frac{\alpha^{2\mu}(r)}{\varepsilon^\mu},
\end{align}
where $\mu = \rho \nu / \lbrace (2\rho +1 ) (\nu + 2) \rbrace$.
Setting $\varepsilon = \alpha^{2\mu/(\mu + 1)}(r)$, we have, by (\ref{eq:proof_u_conv_alpha_deg}), (\ref{eq:proof_u_conv_alpha_deg_R1n}) and (\ref{eq:proof_u_conv_alpha_deg_R2n}), 
\begin{align*}
    |\Exp[h_{2,r}(Z_{r,i_1}, Z_{r, i_2}) h_{2,r}(Z_{r,j_1}, Z_{r, j_2})]| 
    & \lesssim \alpha^{2\mu/(\mu + 1)}(r) + \alpha^{2\mu- 2\mu^2/(\mu+1)}(r)\\
    & = 2 \alpha^{(2 \rho \nu)/(3\rho \nu + 4\rho + \nu +2)}(r),
\end{align*}
uniformly in $\bm{i}, \bm{j}.$ Hence, we obtain by Condition \ref{cond:ser_dep_alpha}
\begin{align*}
&\phantom{{}={}} \frac{m}{\ndb^4} \sum_{(\bm{i}, \bm{j}) \in \JJdb} \big| \Exp[h_{2,r}(\bm{Z}_{r,i_1}, \bm{Z}_{r,i_2}) h_{2,r}(\bm{Z}_{r,j_1}, \bm{Z}_{r,j_2})] \big| \\
& \lesssim m \alpha^{(2 \rho \nu)/(3\rho \nu + 4\rho + \nu +2)}(r) \\
& = \big(m^{3/2+ 2/\nu + 1/(2\rho) + 1/(\rho \nu)} \alpha(r)\big)^{(2 \rho \nu)/(3\rho \nu + 4\rho + \nu +2)}  = o(1),
\end{align*}
as $\# \JJdb = O(m^4)$. 

The sliding blocks version can be treated similarly, see also the proof of Theorem \ref{thm:U_conv}.

Finally, the statements concerning centering at $\tilde \vartheta_r$ follow by the same arguments as in the proof of Corollary \ref{cor:U_conv_bias}.
\end{proof}

\section{Formulas for the asymptotic variances in Section~\ref{subsec:var}}
\label{sec:varformula}

Write $g_j := \Gamma(1-j\gamma)$ for $\gamma < 1/j, j \in \N$ and let $\zeta(3)$ denote Apery's constant.
\begin{lemma}\label{lem:var_est_asyvar}
Let $\gamma < 1/4.$ The asymptotic variances in Corollary~\ref{cor:var_est} can be written as
\begin{align}\label{eq:var_est_asyvar_db}
\sigma^2_{\dbl} = 
\begin{cases}
\frac{4}{\gamma^4} \left( g_4 - 4g_1g_3 -g_2^2 +8g_1^2g_2-4g_1^4 \right), & \gamma \neq 0 \\
\frac{22}{45}\pi^4, & \gamma = 0
\end{cases}
\end{align}
and
\begin{align} \label{eq:var_est_asyvar_sb}
\sigma^2_{\sbl} =  
\begin{cases}
\frac{2}{3\gamma^3} \left( -3g_4I_{2,2} + 8g_1g_3I_{2,1}-6g_1^2g_2I_{1,1} \right), & \gamma > 0 \\
\frac{8}{\gamma^2} \left( \Gamma(-4\gamma)I_{2,2}- 2g_1\Gamma(-3\gamma)I_{2,1}+g_1^2\Gamma(-2\gamma)I_{1,1} \right), & \gamma < 0 \\
2\zeta(3) -48-\frac{8}{3}\pi^2 + \frac{32}{3}\log^3(2) - 48\log^2(2) +96\log(2) + \frac{16}{3} \pi^2 \log(2), & \gamma = 0
\end{cases},
\end{align}
where 
\begin{equation}\label{eq:var_est_asyvar_Iij}
I_{i,k} := \int_0^{1/2} \left(\alpha_{(j+k)\gamma}(w)-1 \right) \big \lbrace w^{-j\gamma-1}(1-w)^{-k\gamma-1} + w^{-k\gamma-1}(1-w)^{-j\gamma-1} \big \rbrace \on{d}\!w    
\end{equation}
 and 
\begin{equation*}
\alpha_\beta\colon (0,1) \to (0,\infty),  \quad w \mapsto \alpha_\beta(w) =
\begin{cases}
\frac{1-(1-w)^{\beta+1}}{w(\beta+1)}, & \beta \neq -1 \\
-\frac{log(1-w)}{w}, & \beta = -1
\end{cases}.
\end{equation*}
\end{lemma}
\begin{proof}
Recall the definition of $h_1$ from (\ref{eq:h_1}). We have $h_1(z) = z^2/2 -\mu_1z+\mu_2/2$ for the variance kernel $h(x,y) = (x-y)^2/2,$ where $\mu_j$ denotes the $j$-th moment of a $\GEV{\gamma}$ distribution. 

\smallskip
\noindent
\textbf{Disjoint case}: Using $\sigma^2_{\dbl} = 4\Var(h_1(Z))$ yields 
\begin{equation} \label{eq:proof_asyvars_sigma_db}
\sigma^2_{\dbl} = \mu_4-\mu_2^2 + 4\mu_1 \left( -\mu_1^3+2\mu_1\mu_2-\mu_3\right).    
\end{equation}
The first four moments of a $\GEV{\gamma}$ distributed random variable are given by
\begin{align*}
&\mu_1 = \frac{g_1-1}{\gamma}, \quad &&\mu_2 = \frac{g_2-2g_1+1}{\gamma^2}, \\
&\mu_3 = \frac{g_3-3g_2+3g_1+1}{\gamma^3}, \quad && \mu_4 = \frac{g_4-4g_3+6g_2-4g_1+1}{\gamma^4}.
\end{align*}
Plugging these into (\ref{eq:proof_asyvars_sigma_db}) gives the result for $\gamma \neq 0.$ The case $\gamma = 0$ is similarly easy and hence omitted.

\textbf{Sliding case}: Let $C_\xi = \Cov(h_1(Z_{1,\xi}),h_1(Z_{2,\xi})).$ First we will consider $\gamma \neq 0.$ A short calculation using the transformation $S_{i,\xi} = (1+\gamma Z_{i,\xi})^{-1/\gamma}$ gives $
C_\xi = 1/4\gamma^4\left(C_{\xi,2,2}-4g_1C_{\xi,2,1}+4g_1^2C_{\xi,1,1} \right),$
where $C_{\xi,j,k} := \Cov(S_{1,\xi}^{-j\gamma}, S_{2,\xi}^{-k\gamma})$ and hence 
\begin{equation}\label{eq:proof_asyvars_sigma_sb}
\sigma^2_{\sbl} = \frac{2}{\gamma^4} \int_0^1 (C_{\xi,2,2}-4g_1C_{\xi,2,1}+4g_1^2C_{\xi,1,1} \on{d}\!\xi.
\end{equation}
Hoeffdings Lemma will be employed to calculate $C_{\xi,j,k},$ thus we need to derive the difference of the joint and product c.d.f.s: To this end use the explicit form of $G_\xi$ for univariate random variables, see e.g. equation (13) in \cite{BucZan23},
\[
G_\xi(x,y) = \exp \left[ 
- \left\lbrace
\xi(1+\gamma x)^{-\frac{1}{\gamma}} + \xi (1+ \gamma y)^{-\frac{1}{\gamma}} + 
(1-\xi) (1+\gamma(x \wedge y)^{-\frac{1}{\gamma}}
\right\rbrace
\right]
\]
for $\xi \in [0,1]$ and $G_\xi(x,y) = G_\gamma(x) G_\gamma(y)$ for $\xi > 1$ to obtain 
\[
\Pm(S_{1, \xi} \le s, S_{2, \xi} \le t) = 1- e^{-s}-e^{-t}+e^{-(s+t)A_\xi(\frac{t}{t+s})}, \qquad s,t>0,
\]
where $A_\xi(w) = \xi + (1- \xi) \{ w \vee (1-w) \}.$ These lead to
\begin{align*}
&\Pm \left(S_{1,\xi}^{-j\gamma} \le s, S_{2,\xi}^{-k\gamma} \le t \right) - \Pm \left( S_{1,\xi}^{-j\gamma} \le s \right) \Pm \left( S_{2,\xi}^{-k\gamma} \le t \right) \\
=\,& \exp\left[ (s^{-\frac{1}{j\gamma}} + t^{-\frac{1}{k\gamma}}) A_\xi \left( \frac{t^{-\frac{1}{k\gamma}}}{s^{-\frac{1}{j\gamma}}+ t^{-\frac{1}{k\gamma}}} \right)\right] - \exp\left[ - \left( s^{-\frac{1}{j\gamma}}+ t^{-\frac{1}{k\gamma}} \right) \right],
\end{align*}
for $s,t > 0, j,k \in \lbrace 1, 2 \rbrace.$ Now by Hoeffding's Lemma:
\begin{align*}
&C_{\xi,j,k} \\ 
=\,&  \int_0^\infty \int_0^\infty \exp\left[ (s^{-\frac{1}{j\gamma}} + t^{-\frac{1}{k\gamma}}) A_\xi \left( \frac{t^{-\frac{1}{k\gamma}}}{s^{-\frac{1}{j\gamma}}+ t^{-\frac{1}{k\gamma}}} \right)\right] - \exp\left[ - \left( s^{-\frac{1}{j\gamma}}+ t^{-\frac{1}{k\gamma}} \right) \right] \on{d}\!s \on{d}\!t.
\end{align*}
Using the substitutions $u = s^{-\frac{1}{j\gamma}} + t^{-\frac{1}{k\gamma}}, w = t^{-\frac{1}{k\gamma}}/ (s^{-1/(j\gamma)}+ t^{-1/(k\gamma)})$ we get 
\begin{equation}\label{eq:proof_asyvars_Cxijk}
C_{\xi,j,k} 
= jk\gamma^2 \int_0^1 \int_0^\infty \left( \e^{-uA_\xi(w)} - \e^{-u}\right) u^{-(j+k)\gamma -1} (1-w)^{-j\gamma-1} w^{-k\gamma -1} \on{d}\!u \on{d}\!w    
\end{equation}
Distinguish cases, first let $\gamma < 0$. Use the fact that 
\[
\int_0^\infty \left( \e^{-zt}-\e^{-u} \right) u^{-\beta-1} \on{d}\!u 
= \Gamma(-\beta)\left(z^\beta-1\right), \qquad z >0, \beta < 0
\]
to obtain 
\begin{equation}\label{eq:proof_asyvars_Cxijk<0}
C_{\xi,j,k} = jk\gamma^2 \int_0^1\Gamma(-(j+k)\gamma) \left( A_\xi^{(j+k)\gamma}-1 \right) (1-w)^{-j\gamma -1}w^{-k\gamma -1} \on{d}\!w.
\end{equation}
Note $A_\xi(w) = A_\xi(1-w)$ for $w \in (0,1)$ and recall the definition of $\alpha_\beta(w).$ For $w \in (0,1)$ and $\beta > 0$ it then follows $\int_0^1(\xi w+1-w)^\beta \on{d}\!\xi = \alpha_\beta(w).$ This in conjunction with the symmetry of $A_\xi$ and (\ref{eq:proof_asyvars_Cxijk<0}) yields 
\begin{align}
\begin{split}    
&\int_0^1C_{\xi,j,k} \on{d}\!\xi \\
=  \, & jk\gamma^2 \Gamma(-(j+k)\gamma) \int_0^{1/2} \left(\alpha_{(j+k)\gamma}(w)-1 \right) \big \lbrace w^{-j\gamma-1}(1-w)^{-k\gamma-1} \\ 
& \hspace{4cm} + w^{-k\gamma-1}(1-w)^{-j\gamma-1} \big \rbrace \on{d}\!w.
\end{split}
\end{align}
Recall the definition of $I_{j,k}$ in (\ref{eq:var_est_asyvar_Iij}) and use (\ref{eq:proof_asyvars_sigma_sb}) to obtain 
\[
\sigma^2_{\sbl} = \frac{8}{\gamma^2}\left(\Gamma(-4\gamma)I_{2,2}-2g_1^2\Gamma(-3\gamma)I_{2,1}+\Gamma(-2\gamma)I_{1,1} \right), \qquad \gamma < 0.
\]
Consider now $\gamma > 0$ and note 
\[
\int_0^\infty \left( \e^{-zt}-\e^{-u} \right) u^{-\beta-1} \on{d}\!u 
= \frac{\Gamma(1-\beta)}{\beta} (1-z^\beta), \qquad \beta \in (0,1).
\]
to obtain via (\ref{eq:proof_asyvars_Cxijk})
\begin{align*}
C_{\xi,j,k} = 
& -\gamma\frac{jk  g_{j+k}}{j+k} \int_0^1\left(A_\xi^{(j+k)\gamma}-1 \right) w^{-j\gamma-1}(1-w)^{-k\gamma -1} \on{d}\!w \\
= \,& -\gamma\frac{jk  g_{j+k}}{j+k} I_{j,k}.
\end{align*}
Plugging this into (\ref{eq:proof_asyvars_sigma_sb}) yields the in (\ref{eq:var_est_asyvar_sb}) stated formula for $\gamma > 0.$

Let $\gamma = 0$ and use the transformation $S_{i,\xi} = \exp(-Z_{i,\xi})$ to obtain 
\begin{equation}\label{eq:proof_asyvars_Cxi_gumb}
  C_\xi :=  \Cov(h_1(Z_{1,\xi}), h_1(Z_{2,\xi}))
= \frac{1}{4}C_{\xi,2,2} + \gamma_E C_{\xi,2,1} + \gamma_E^2C_{\xi,1,1},  
\end{equation}
where $C_{\xi,j,k} := \Cov(\log^jS_{1,\xi}, \log^kS_{2,\xi}), j,k=1,2$ and $\gamma_E$ denotes the Euler–Mascheroni constant. Simple but tedious calculations to derive the differences needed in Hoeffdings Lemma result in 
\begin{align*}
&\Pm(\log S_{1,\xi} \leq s, \log S_{2,\xi} \leq t) - \Pm(\log S_{1,\xi} \leq s) \Pm(\log S_{2,\xi} \leq t) \\
= \,&\exp \left(-(\e^s+ \e^t)A_\xi \left(\frac{\e^t}{\e^t+\e^s}\right)\right) - \exp\left(-(\e^s+\e^t)\right), \qquad s,t \in \R; \\
& \Pm(\log^2 S_{1,\xi} \leq s, \log S_{2,\xi} \leq t) - \Pm(\log^2 S_{1,\xi} \leq s) \Pm(\log S_{2,\xi} \leq t) \\
= \, & \exp \left(-(\e^{\sqrt{s}} + \e^t) A_\xi\left( \frac{\e^t}{\e^t+\e^{\sqrt{s}}}\right) \right)  - \exp \left(-(\e^{-\sqrt{s}} + \e^t) A_\xi\left( \frac{\e^t}{\e^t+\e^{-\sqrt{s}}}\right) \right) \\
& \hspace{0.5cm} - \exp(-(\e^{\sqrt{s}} + \e^t)) + \exp(-(\e^{-\sqrt{s}} + \e^t)), \qquad s > 0, t \in \R; \\
& \Pm(\log^2 S_{1,\xi} \leq s, \log^2 S_{2,\xi} \leq t) - \Pm(\log^2 S_{1,\xi} \leq s) \Pm(\log^2 S_{2,\xi} \leq t) \\
= \,& \exp \left(-(\e^{\sqrt{s}} + \e^{\sqrt{t}}) A_\xi\left( \frac{\e^{\sqrt{t}}}{\e^{\sqrt{t}}+\e^{\sqrt{s}}}\right) \right)
 - \exp \left(-(\e^{\sqrt{s}} + \e^{\sqrt{t}})\right)\\
& \hspace{0.5cm} + \exp \left(-(\e^{-\sqrt{s}} + \e^{-\sqrt{t}}) A_\xi\left( \frac{\e^{-\sqrt{t}}}{\e^{-\sqrt{t}}+\e^{-\sqrt{s}}}\right) \right) - \exp \left(-(\e^{-\sqrt{s}} + \e^{-\sqrt{t}})\right)\\
& \hspace{0.5cm} -\exp \left(-(\e^{-\sqrt{s}} + \e^{\sqrt{t}}) A_\xi\left( \frac{\e^{\sqrt{t}}}{\e^{\sqrt{t}}+\e^{-\sqrt{s}}}\right) \right) + \exp \left(-(\e^{-\sqrt{s}} + \e^{\sqrt{t}})\right)\\
& \hspace{0.5cm} -\exp \left(-(\e^{\sqrt{s}} + \e^{-\sqrt{t}}) A_\xi\left( \frac{\e^{-\sqrt{t}}}{\e^{-\sqrt{t}}+\e^{\sqrt{s}}}\right) \right) + \exp \left(-(\e^{\sqrt{s}} + \e^{-\sqrt{t}})\right), \, s,t > 0.
\end{align*}
Hoeffding's Lemma, the upper displays and substitutions of the form $u = \e^{\pm \sqrt{s}} + \e^{\pm \sqrt{t}}, w = \e^{\pm \sqrt{t}} / (\e^{\pm \sqrt{s}} + \e^{\pm \sqrt{t}})$ yield
\begin{align} \label{eq:proof_asyvars_Cxijk_gumb}
\begin{split}
& C_{\xi, 1, 1} = \int_0^1 \int_0^\infty \left(\e^{-uA_\xi(w)} - \e^{-u} \right) \frac{1}{uw(1-w)} \on{d}\!u \on{d}\!w, \\
& C_{\xi, 2, 1} = 2 \int_0^1 \int_0^\infty \left(\e^{-uA_\xi(w)} - \e^{-u} \right) \frac{\log u + \log(1-w)}{uw(1-w)} \on{d}\!u \on{d}\!w,\\
& C_{\xi, 2, 2} = 4 \int_0^1 \int_0^\infty \left(\e^{-uA_\xi(w)} - \e^{-u} \right) \frac{(\log u + \log(1-w))(\log u + \log w)}{uw(1-w)} \on{d}\!u \on{d}\!w.
\end{split}
\end{align}
Invoke the following integral identities 
\begin{align*}
\begin{split}
&\int_0^\infty \left( \e^{-uz} - e^{-u} \right) \frac{1}{u} \on{d}\!u = -\log z, \\
& \int_0^\infty \left( \e^{-uz} - e^{-u} \right) \frac{\log u}{u} \on{d}\!u = \log z \frac{\log z+ 2\gamma_E}{2},\\
& \int_0^\infty \left( \e^{-uz} - e^{-u} \right) \frac{\log^2 u}{u} \on{d}\!u = 
-\log z \frac{\pi^2+6\gamma^2_E+2\log^2z+6\gamma_E\log z}{6}
\end{split}
\end{align*}
for $z > 0$ to obtain via (\ref{eq:proof_asyvars_Cxijk_gumb})
\begin{align}
\begin{split}
&\int_0^1 C_{\xi,1,1} \on{d}\!\xi = \int_0^1 \frac{1}{w(1-w)} \int_0^1 - \log(A_\xi(w)) \on{d}\!\xi \on{d}\!w,\\ 
&\int_0^1 C_{\xi,2,1} \on{d}\!\xi \\ 
=\,&  \int_0^1 \frac{1}{w(1-w)} \int_0^1 \log(A_\xi(w)) \big[ 2\gamma_E - 2\log(1-w) \\ 
& \hspace{0.5cm} +\log A_\xi(w) \big] \on{d}\!\xi  \on{d}\!w ,\\ 
&\int_0^1 \frac{C_{\xi,2,2}}{4} \on{d}\!\xi \\ 
= \,& \int_0^1 \frac{1}{w(1-w)} \int_0^1 -\gamma^2_E \log A_\xi(w) + \gamma_E \log A_\xi(w) \big[ -\log A_\xi(w) \\
& \hspace{0.5cm} +\log(1-w) +\log w \big] \\ 
& \hspace{0.5cm}  +\log A_\xi(w) \Big[ -\frac{\pi^2+2\log^2 A_\xi(w)}{6} -  \log w \log(1-w)\\
& \hspace{1cm} + \log A_\xi(w) \frac{\log(1-w) + \log w}{2} \Big] \on{d}\!\xi  \on{d}\!w .
\end{split}
\end{align}
Now plug these formulas into (\ref{eq:proof_asyvars_Cxi_gumb}), simplify and integrate to get 
\begin{align*}
&\int_0^1 C_\xi \on{d}\!\xi = 
\int_0^{1/2} \frac{1}{w(1-w)} \int_0^1 \log A_\xi(w) \Big[ -\frac{\pi^2+2\log^2A_\xi(w)}{3}  \\
& \hspace{0.5cm} +\log A_\xi(w) (\log(1-w)+\log w) - 2\log w \log(1-w) \Big] \on{d}\!\xi \on{d}\!w \\
= \, & \int_0^{1/2} \frac{1}{3w^2(1-w)} \Big[ (w-1)\log^3(1-w)-3(w-1)\log^2(1-w)\log w \\
& \hspace{0.5cm} +w(12+\pi^2+6\log w) \\ 
& \hspace{0.5cm} + \log(1-w) (12+\pi^2-w(6+\pi^2)+6\log w) \Big] \on{d}\!w \\
=: \, &\int_0^{1/2} c_\xi \on{d}\!\xi.
\end{align*}
The integrand $c_\xi$ has the antiderivative 
\begin{align*}
B(w) =\, & - \frac{1}{3w} \Big[ 6w\on{Li}_2(w) +6w\on{Li}_3(1-w)(\log(1-w)-1) \\
& \hspace{0.5cm} + w\log^3(1-w) - \log^3(1-w)-3w\log w \log^2(1-w) \\ 
& \hspace{0.5cm} +3\log^2(1-w) \log w +3\log^2(1-w)-6w\log(1-w) \\
& \hspace{0.5cm} +6w \log(1-w)\log w + 6 \log w \log(1-w) \\ 
& \hspace{0.5cm} +\pi^2 \log(1-w) + 18 \log(1-w) + 6w\log w
\Big], \quad w \in (0,1/2),
\end{align*}
where $\on{Li}_j(w) := \sum_{k=1}^\infty w^k/k^j$ denotes the Polylogarithm function for $w \in (0,1)$ and $j \in \N.$ 
Take the limits to get 
\begin{align*}
\lim_{w \downarrow 0} F(w) =\, & 6- 2\zeta(3),  \\ 
\lim_{w \uparrow 0} F(w) =\, &-\frac{7}{4}\zeta(3)-\frac{\pi^2}{3} + \frac{4}{3}\log^3(2) - 6\log^2(2)+12\log2 + \frac{2}{3} \pi^2 \log 2,
\end{align*}
which imply the formula in (\ref{eq:var_est_asyvar_sb}) for $\gamma = 0$.
\end{proof}
\end{document}